\documentclass[12pt]{article}
\usepackage{epsfig,amsmath,amssymb,graphicx}
\usepackage{comment,fullpage}
\usepackage{authblk}
\usepackage{multirow}
\usepackage{hhline}

\newtheorem{lemma}{Lemma}[section]
\newtheorem{theorem}[lemma]{Theorem}
\newtheorem{proposition}[lemma]{Proposition}

\newtheorem{example}[lemma]{Example}

\newcommand {\qed}{\rule{2mm}{2mm}\medskip}
\newenvironment {proof}{\noindent{\bf Proof:}~}{~\qed}
\def\ZZ{{\mathbb Z}}

\newcommand{\br}[1]{\{#1\}}

\begin{document}

% THE TITLE PAGE 

\title{ Tight Heffter Arrays Exist
for all Possible Values: \\ The Research Report}

\author[1]{Dan S.\ Archdeacon}
\author[2]{Tomas Boothby}
\author[1]{Jeffrey H.\ Dinitz}
\affil[1]{Dept. of Mathematics and Statistics,
University of Vermont,
Burlington, VT 05405
U.S.A.}
\affil[2]{Dept. of Mathematics,
     Simon Fraser University,
     Burnaby, BC \ V5A\ 1S6  Canada}

%\vspace*{1.0ex}
\begin{comment}
\begin{center}
\begin{minipage}[t]{.45\textwidth}
   \begin{center}
     Dan Archdeacon \\
     Dept. of Math. and Stat. \\
     University of Vermont \\
     Burlington, VT 05405 \ \ USA\\
   \end{center}
 \end{minipage}\hspace*{.1\textwidth}
 \begin{minipage}[t]{.45\textwidth}
   \begin{center}
     Tom Boothby \\
     Dept. of Math. \\
     Simon Fraser University \\
     Burnaby, BC \ V5A\ 1S6  Canada\\
     {\tt tboothby@sfu.ca}
   \end{center}
 \end{minipage}
\end{center}
\vspace*{1.0ex}

\begin{center}
     Jeff Dinitz \\
     Dept. of Math. and Stat. \\
     University of Vermont \\
     Burlington, VT 05405 \ \ USA\\
     {\tt jeff.dinitz@uvm.edu}
\end{center}

\vspace*{1.0ex}
\begin{center}
{\bf Draft: Not for Distribution}\\
%\vspace*{1.0ex}
\today
\end{center}
\end{comment}

\maketitle
%\vspace*{1.0ex}
%\center{\today}

\begin{abstract} A tight Heffter array $H(m,n)$ is an $m \times n$ matrix with nonzero entries from $\mathbb{Z}_{2mn+1}$ such that {\em i}) the sum of the elements in each row and each column is 0, and {\em ii}) no element from $\{x,-x\}$ appears twice. We prove that $H(m,n)$ exist if and only if both $m$ and $n$ are at least 3. If all entries are integers of magnitude at most $mn$ satisfying every row and column sum is 0 over the integers and also satisfying $ii$) we call $H$ an integer Heffter array. We show integer Heffter arrays exist if and only if $mn \equiv 0,3 \pmod{4}$.  Finally, an integer Heffter array is shiftable if each row and column contains an the same number  of positive and negative integers. We show that shiftable integer arrays exists exactly when both $m,n$ are even.

This research report contains all of the details of the proofs.  It is meant to accompany the paper \cite{ABD}.
\end{abstract}

%%%%%%%%%%%%%%%%%%%%%%%%%%%%%%%%%%%%%%%%%%%%%%
\section {Introduction}\label{introduction}
%%%%%%%%%%%%%%%%%%%%%%%%%%%%%%%%%%%%%%%%%%%%%%

We begin with the general definition of Heffter arrays from \cite{A}.  A Heffter array $H(m,n;s,t)$ is an $m \times n$ matrix with nonzero entries from $\mathbb{Z}_{2ms+1}$ such that 
\begin{enumerate}
  \item 
each row contains $s$ filled cells and each column contains $t$ filled cells,
\item the elements in every row and column sum to 0 in $\mathbb{Z}_{2ms+1}$, and 
\item  for every $x \in \mathbb{Z}_{2ms+1} \setminus \{0\}$, either $x$ or $-x$ appears in the array.
\end{enumerate}

Obviously if $H$ is a Heffter array $H(m,n;s,t)$, then $H^T$ is a Heffter array $H(n,m;t,s)$.   The notion of a Heffter array $H(m,n;s,t)$ was first defined by Archdeacon in \cite{A}.  It is shown there that a Heffter array with a pair of special row and column orderings can be used to construct an embedding of the complete graph $K_{2ms+1}$ on a surface. This connection is given in the following theorem.

\begin{theorem}\label{heffter1}{\rm \cite{A}} Given a Heffter array $H(m,n;s,t)$ with compatible orderings $\omega_r$ of the symbols in the rows of the array and $\omega_c$ on the symbols in the columns of the array, then there exists an embedding of $K_{2ms+1}$ on an orientable surface such that every edge is on a face of size $s$ and a face of size $t$. Moreover, if $\omega_r$ and $\omega_c$ are both simple, then all faces are simple cycles. \end{theorem}

Such an embedding as described in Theorem \ref{heffter1} is called an $\{s,t\}$-biembedding.  We refer the reader to \cite{A} for the definition of a simple ordering and the definition of compatible orderings and to \cite{GG} for a survey containing information on biembeddings of cycle systems.  In   \cite{ADS} we address the ordering problem in more detail. 
We also note that using Theorem \ref{heffter1} and the $3 \times n$  Heffter arrays constructed in this paper, Dinitz and Mattern \cite{DM}  proved that for every $n$, there exists a $\{3,n\}$-biembedding of the complete graph on $6n +1 $ vertices on an orientable surface.
 
In this paper we will not concern ourselves with the ordering problem or the embedding problem and will just concentrate on the construction of the Heffter arrays (which we feel are very interesting combinatorial objects in their own right).  In this paper we will concentrate on constructing Heffter arrays where $n=s$ (and necessarily $m=t$).  These arrays have the property that they have no empty cells.  We call these arrays {\em tight} Heffter arrays and for the remainder of this paper we will use the notation $H(m,n)$ to denote a tight Heffter array $H(m,n;n,m)$.  
In a companion paper \cite{ADDY}, we consider {\em square} Heffter arrays with empty cells, i.e. Heffter arrays $H(n,n;k,k)$ where $3\leq k < n$.

We now reiterate the definition of Heffter arrays that we will be using for the reminder of this paper.
A {\em Heffter array} $H(m,n)$ is an $m \times n$ array with entries from $\mathbb{Z}_{2mn+1}$ such that 1) every cell is filled, 2) every row and column sums to 0 modulo $2mn+1$, and 3) no element from $\{x,-x\}$ appears twice. For the third condition it is convenient to use modular representatives $-mn, -mn+1,\dots,-1,0,1,\dots, mn$. Note that if a Heffter array $H$ exists, then necessarily $m,n \ge 3$, also  the element 0 is not used in $H$. Heffter arrays represent a type of magic square where each number from the set $\{1,\dots,mn\}$ is used once up to sign. In Example \ref{first_example} we give  examples of two small  Heffter arrays. 

\begin{example}\label{first_example} A Heffter array $H(3,3)$ over $\mathbb Z_{19}$ and a  Heffter array $H(3,4)$ over $\mathbb Z_{25}$.
$$\begin{array}{|c|c|c|}\hline
-8 & -2 & -9 \\  \hline
7 & -3 & -4 \\ \hline
1 & 5 & -6 \\ \hline
\end{array} \hspace{1in}
\begin{array}{|c|c|c|c|} \hline
1 & 2 & 3 & -6 \\ \hline
8 & -12 & -7 & 11 \\ \hline
-9 & 10 & 4 & -5 \\ \hline
\end{array}
$$
%\caption{\label{first_example} A Heffter array $H(3,4)$ over $\mathbb Z_{25}$}
\end{example}
 
\medskip
Our main goal in this paper is to prove the existence of Heffter arrays $H(m,n)$ for all necessary values of $m$ and $n$.
In the course of the proof two types of Heffter arrays will be especially helpful. An {\em integer Heffter array}  is a Heffter array $H(m,n)$ with integer entries from $\{-mn, \dots, mn\}$ such that 1) every row and column sums to 0 over the integers, and 2) no element from $\{x,-x\}$ appears twice. The $H(3,4)$ in Example \ref{first_example} is  an integer Heffter array while the $H(3,3)$ is not an integer Heffter array.

Let $H(m,n)$ be an arbitrary $m \times n$ array with integer entries $a_{i,j}$. The {\em support} of $H$ is the multiset $sup(H) = |a_{i,j}|$ as $i,j$ range over all entries. For example, if $H$ is an integer Heffter array, then $sup(H) = \{1,\dots,mn\}$. Define the {\em shift of $H$ by $k$ units} as the matrix $H \pm k$ whose $(i,j)$ entry is given by 
$$(H \pm k) (i,j)= \bigg\{ \aligned a_{i,j} + k, &\quad\quad \mbox{when }a_{i,j} > 0 \\ a_{i,j} - k, &\quad\quad \mbox{when }a_{i,j} < 0. \endaligned$$ 
We call  any  array of integers (in particular integer Heffter arrays) {\em shiftable} if it has the same number of positive entries as negative entries in each row and column. Denote an $m \times n$ shiftable Heffter array as $H_s(m,n)$.  If $A$ is a shiftable integer array, then $A \pm k$ has the same row and column sums as $A$. In particular if $H(m,n)$ is a shiftable integer Heffter array, then $H \pm k$ satisfies all requirements for an integer Heffter array, except that $sup(H\pm k) = \{k+1,\dots,k+mn\}$ where $sup(H)=\{1,\dots,mn\}$. Shiftable arrays provide very useful ingredients in building  Heffter arrays (see also \cite{ADDY}).

The following lemma gives necessary conditions for the existence of $H(m,n)$ and for the existence of $H(m,n)$ and  $H_s(m,n)$.

\begin {lemma}\label{necessary}  If there exists an  $H(m,n)$, then necessarily $m,n\geq 3$.  If there is an integer  $H(m,n)$, then  $mn \equiv 0,3$ (mod 4).  Furthermore,
if there exists an integer $H_s(m,n)$, then necessarily $m$ and $n$ are even (and so $mn \equiv 0$ (mod 4)).  \end{lemma}  

\begin{proof}  Let $H$ be a Heffter array $H(m,n)$. First see that $m$ and $n$ are not 1 or 2 since $0 \notin H$ and for any $x$, both $x$ and $-x$ are not in $H$.  Now, assume the H is an integer Heffter array.  In order for each row of $H$ to sum to zero, each row must contain an even number of odd numbers.  Hence the entire array contains an even number of odd numbers.   Now, the support of  the $H(m,n)$ is the  set $S =\{1,2,  \ldots ,mn\}$.  There will be an even number of odd numbers in $S$ exactly when $mn \equiv 0,3$ (mod 4).

If H is an integer $H_s(m,n)$, then clearly $m$ and $n$ must be even in order to have the same number of positive and negative entries in each row and each column.  It follows that $mn \not\equiv 3$ (mod 4), hence $mn \equiv 0$ (mod 4). 
\end{proof}

In the following example we give shiftable Heffter arrays for some small orders.

\begin {example} \label{three_evens}
Three shiftable integer Heffter arrays: $H_s(4,4)$, $H_s(4,6)$, and $H_s(6,6)$.

\renewcommand{\arraycolsep}{2pt}
\begin{center}
% left side
\begin{minipage}{4.0cm}
{$\begin{array}{|c|c|c|c|} \hline
1&-2&-3&4 \\ \hline
-5&6&7&-8 \\ \hline
9&-10&-11&12 \\ \hline
-13&14&15&-16 \\ \hline
\end{array}$}
\end{minipage} 
% kiddle
\quad\quad
\begin{minipage}{6.0cm}
{$ \begin{array}{|c|c|c|c|c|c|} \hline
1&-2&3&-4&11&-9 \\ \hline
-7&8&-12&10&-5&6 \\ \hline
-13&14&-15&16&-23&21 \\ \hline
19&-20&24&-22&17& -18\\ \hline
\end{array}
$}
\end{minipage} 
\end{center}
% third square
\begin{center}
\begin{minipage}{6.0cm}
{$ \begin{array}{|c|c|c|c|c|c|} \hline
-1&5&2&-7&-9&10 \\ \hline
3&-4&-6&8&11&-12 \\ \hline
-21&22&-13&17&14&-19 \\ \hline
23&-24&15&-16&-18&20 \\ \hline
26&-31&-33&34&-25&29 \\ \hline
-30&32&35&-36&27&-28 \\ \hline
\end{array}
$}
\end{minipage} 
% tidy up figure
\end{center}
\end{example}
\renewcommand{\arraycolsep}{6pt}

We now give a full statement of the main result of this paper.

\begin{theorem} Let $m,n$ be integers at least 3. Then
\begin{enumerate}
\item There exists a Heffter array $H(m,n)$, 
\item There is an integer Heffter array $H(m,n)$ if and only if $mn \equiv 0,3 \pmod{4}$, and
\item There is a shiftable Heffter array $H_s(m,n)$ if and only if both $m$ and $n$ are  even.
\end{enumerate}
\end{theorem}

The proof will consist of several cases.  In Section \ref{evens} we consider the case when both $m$ and $n$ are even.  In Section \ref{3byn} we cover the case when $m=3$ and in Section \ref{5byn} we do the case of $m=5$. In both of those sections it is necessary to consider 8 cases for $n$ depending on its congruence modulo 8.  Next we consider the cases when $m$ is odd and $n$ is even, this is covered in Section \ref{extending}.  Finally, in Section \ref{ells} we use an L-construction to tackle the case when both $m$ and $n$ are odd.  We summarize the cases again in the conclusion, Section \ref{Conclusion}.

%%%%%%%%%%%%%%%%%%%%%%%%%%%%%%%%%%%%%%%%%%%%%%
\section {H($m,n$) with both $m$ and $n$ even}\label{evens}
%%%%%%%%%%%%%%%%%%%%%%%%%%%%%%%%%%%%%%%%%%%%%%

In this section we construct Heffter arrays $H(m,n)$ for both $m$ and $n$ even. These arrays are all both integer and shiftable. Example \ref{three_evens} gives the three smallest examples.
The desired $H(m,n)$ will be constructed using these three small arrays as subarrays, called {\em tiles}. The entries in a tile may be shifted so as to keep all entries in the entire array unique up to sign.

\begin{theorem}\label{evenbyeventhm} If $m$  and $n$ are even numbers with $m,n \geq 4$, then there exists a shiftable integer $H_s(m,n)$. \end{theorem}

\begin{proof} We begin by assuming that $m \equiv n \equiv 0 \pmod{4}$. First form a $4 \times n$ array $A$ by horizontally juxtaposing $n/4$ different $4 \times 4$ arrays $A_i$
$$A = \begin{array}{|c|c|c|c|c|} \hline \stackrel{\ }{A_1} & A_2 & A_3 & \dots & A_{n/4} \\ \hline \end{array}.$$
Let $H$ be the shiftable $4 \times 4$ array given in Example \ref{three_evens}. Set $A_i = H \pm 16(i-1)$. Observe each row and column of $A$ sum to 0, since they do in each subarray. Also by our choice of shifts $sup(A) = \{1,\dots,4n\}$. Hence $A$ is a shiftable Heffter array $H_s(4,n)$. To construct our desired $H(m,n)$, we vertically stack $m/4$ copies of $A$ with the $i^{th}$ copy shifted to $A \pm 4n(i-1)$. Note that the final array is shiftable.\medskip

Next assume that $m \equiv 0 \pmod{4}$ and $n \equiv 2 \pmod{4}$. Again form the $4 \times n$ array
$$A = \begin{array}{|c|c|c|c|c|} \hline %\stackrel{\ }
{B} & A_1 & A_2 & \dots & A_{(n-6)/4} \\ \hline \end{array}.$$
Set $B$ equal to the shiftable $H_s(4,6)$  and set each $A_i$ to the shiftable $H_s(4,4)$, both from Example \ref{three_evens}. By suitable shifts on $B$ and the $A_i$'s as before we can make $sup(A) = \{1,\dots,4n\}$. To construct our desired $H_s(m,n)$ we vertically stack $m/4$ copies of $A$ appropriately shifted.\medskip

Finally we assume that $m \equiv n \equiv 2 \pmod{4}$. Construct the $6 \times n$ array
$$A = \begin{array}{|c|c|c|c|c|} \hline %\stackrel{\ }
{C} & B_1^T & B_2^T & \dots & B_{(n-6)/4}^T \\ \hline \end{array}.$$
Set $C$ to be the  $H_s(6,6)$   and set each $B_i^T$ as the transpose of the  $H_s(4,6)$, both from Example \ref{three_evens}. By suitable shifts on $C$ and $B_i$'s we can make $sup(A) = \{1,\dots,6n\}$. To construct our desired $H(m,n)$ we fill in the first 6 rows with the $6 \times n$ shiftable array just constructed. The remaining empty cells form an $(m-6) \times n$ array which can be filled with a shiftable array from an earlier modulus case.  \end{proof}

%%%%%%%%%%%%%%%%%%%%%%%%%%%%%%%%%%%%%%%%%%%%%%
\section { ${3 \times n}$ Heffter arrays}\label{3byn}
%%%%%%%%%%%%%%%%%%%%%%%%%%%%%%%%%%%%%%%%%%%%%%

The cases where one or both of $m$ and $n$ is odd are considerably more difficult. In this section we start the investigation by constructing $3 \times n$ Heffter arrays. These arrays will be integer when $n$ satisfies the necessary condition $n \equiv 0,1 \pmod{4}$. By the nature of the construction it is convenient to consider the residue of $n$ modulo 8.  We begin with the case of  $H(3,n)$ when $n\equiv 0 \pmod 8$ to illustrate the construction technique.  

%%%%
\subsection{ \label{3by8} Constructing an  H(3,{\em n}) with %$\mathbf 
{\em n} $\equiv$ 0 (mod 8)}
%%%%%

\begin{comment}

We prove the first of 8 cases to illustrate our technique. 

\begin{comment}
\begin{figure}
\begin{center}
$S^0 = 
\begin{array}{|c|c|c|c|} \hline
-12k-13 & -10k-11 & 4k+6 & 4k+3 \\ \hline
4k+4 & -8k-7 &  18k+17 & 18k+19 \\ \hline
8k+9 & 18k+18 & -22k-23 & -22k-22 \\ \hline
\end{array} $
\end{center}
\begin{center}
$T_r^0  = 
\begin{array}{|c|c|c|c|} \hline
8k+r+10 & -8k+2r-8 & 14k-r+14 & -4k+2r-1 \\ \hline
8k-2r+5 & -16k-r-16 & -4k+2r-2 & -18k-r-20 \\ \hline
-16k+r-15& 24k-r+24 & -10k-r-12 & 22k-r+21 \\ \hline
\end{array} $
\end{center}
\caption{\label {3by8tiles} Two $3 \times 4$ tiles ($k = (n-8)/8$)}
\end{figure}
\end{comment}

\begin{proposition}\label{3by8_theorem} There exists a $3 \times n$ integer Heffter array for all $n \equiv 0 \pmod{8}$.\end{proposition}

\begin{proof} We write $n=8k+8$ with $k \geq 0$.  Our desired array will be built from $k+1$ {\em tiles} each of which is a $3 \times 8$ subarray (for typesetting purposes we write the transpose of each tile here and throughout the remainder of the paper). It has a sporadic tile $A$ and a family of tiles $A_r$, $r = 0,\ldots,k-1$ where 

\[ A^T = \left[\begin{smallmatrix}
-12k - 13&4k + 4&8k + 9\\
-10k - 11&-8k - 7&18k + 18\\
4k + 6&18k + 17&-22k - 23\\
4k + 3&18k + 19&-22k - 22\\
10k + 10&4k + 5&-14k - 15\\
-4k - 8&-18k - 16&22k + 24\\
12k + 14&-2&-12k - 12\\
-1&-20k - 20&20k + 21\end{smallmatrix}\right] \mbox{ and }
A_r^T = \left[\begin{smallmatrix}
8k + 2r + 10&8k - 4r + 5&-16k + 2r - 15\\
-8k + 4r - 8&-16k - 2r - 16&24k - 2r + 24\\
14k - 2r + 14&-4k + 4r - 2&-10k - 2r - 12\\
-4k + 4r - 1&-18k - 2r - 20&22k - 2r + 21\\
-8k - 2r - 11&-8k + 4r - 3&16k - 2r + 14\\
8k - 4r + 6&16k + 2r + 17&-24k + 2r - 23\\
-14k + 2r - 13&4k - 4r&10k + 2r + 13\\
4k - 4r - 1&18k + 2r + 21&-22k + 2r - 20\end{smallmatrix}\right].
\]

The row sums of $A$ are $[4k,\, -2k,\, -2k]$ and the column sums of $A$ are 
$[0,\, 0,\, 0,\, 0,\, 0,\, 0,\, 0,\, 0]$. For each $0\leq r\leq k-1$ we see that the row sums of  
$A_r$ are $[-4,\, 2,\, 2]$ while the column sums are $[0,\, 0,\, 0,\, 0,\, 0,\, 0,\, 0,\, 0]$.
So each of these tiles are column Heffter but not row Heffter. Now concatenate these tiles into a $3 \times n$ array 
$$ H =\begin{array}{|c|c|c|c|c|c|} \hline 
A & A_0& A_1& \cdots &   A_{k-2}& A_{k-1} \\ \hline \end{array}.$$

We claim that $H$ is an integer Heffter array $H(3,n)$. First, since each tile is column Heffter, so is $H$. Second, the sum of the rows of the $A_r$'s is $k \times [-4,\, 2,\, 2] =[-4k,\, 2k,\, 2k] $ which when added to the row sums in $A$ gives that every row sum is indeed 0, as desired.

It only remains to check that the support of $H$ is indeed the set $\{1,2, \ldots, 24k + 24\}=\{1,2,\ldots 3n\}$.  In Section \ref{justify} we will explain how this check was coded into a program in Sage and automated.  We will return to this particular case after that discussion.
\end{proof} 

\subsection {\label{construction} Some history and how the tiles were constructed}
%%%%%

Proposition \ref{3by8_theorem} establishes the existence of integer $H(3,n)$ for $n \equiv 0 \pmod{8}$ using a difference method displayed as tiles. But where did these tiles come from? In this subsection we discuss the relationship between these Heffter arrays and the first Heffter difference problem as well as with the Skolem sequences underlying their construction. This is actually a very abbreviated discussion - far more details can be found in \cite{Tom-thesis}.  

 {\em  Heffter's first difference problem} asks if the integers $\{1,\dots,3n\}$ can be partitioned into $n$ triples $\{a,b,c\}$ such that either $a + b = c$ or $a + b + c = 6n+1$. If so, then the set of all triples $T = \{0, a,a+b\}$ is a $(6n+1,3,1)$ difference family, that is, if each the $n$ triples $T$ are developed modulo $6n+1$, the result is a Steiner triple system. The columns of a $3 \times n$ Heffter array form a solution to the first Heffter difference problem. If the Heffter array is integer, then all column sums satisfy $a + b = c$ for a suitable ordering of the entries.  A solution to Heffter's first difference problem was given by Peltesohn \cite {Pelt} in 1939.

A {\em Skolem sequence} of order $n$ is a sequence $S = (s_1,s_2,\dots,s_{2n})$ of $n$ integers satisfying 1) every number $1 \le k \le n$ appears exactly twice, first in position $\ell(k)$ and second in position $r(k)$, and 2) for all $k$, $r(k) - \ell(k) = k$. An example of a Skolem sequence of order 5 is $S_5 = (1,1,3,4,5,3,2,4,2,5)$. Skolem sequences are known to exist if and only if $n \equiv 0,1 \pmod{4}$ (see  \cite{X1}). 

Skolem sequences can be used to construct a solution to the first Heffter difference problem and hence provide candidate columns for an $H(3,n)$. Given a Skolem sequence of order $n$, for each $k$ with $1\leq k \leq n$, form a triple $T_k =\{a_k,b_k,c_k\} = \{k,\ell(k)+n,r(k)+n\}$. The $a_k$'s range over $\{1,\dots,n\}$ while $\ell(k)+n,r(k)+n$ range over $\{n+1,\dots,3n\}$. Moreover, $k + (\ell(k) + n) = (r(k) +n)$, so these triples solve Heffter's first difference problem on the set $\{1,\dots,3n\}$. 

A suitable Skolem sequence  will form the basis for the columns of our $3 \times n$ Heffter array. But each possible column can be placed in the array in 12 different ways (order the elements in 6 ways and/or negate all). We need to order each of the columns so that the row sums are zero.

\begin{comment}
\begin{figure}
\noindent $\begin{array}{|cc|cc|} \hline
x_1 + 2j &  -x_1-2j-1 & -x_2-2j & x_2 + 2j + 1 \\ \hline
-y_1 + 2j &  y_1-2j-1 & y_2-2j & -y_2 + 2j + 1  \\ \hline
y_1-x_1-4j & x_1 - y_1 + 4j + 2 & x_2-y_2 + 4j & y_2 -x_2 -4j -2 \\ \hline
\end{array}$
$\cdots$

\medskip

\hspace*{.1\textwidth}$\cdots$\begin{minipage}{6.0cm}
{$ \begin{array}{|cc|cc|} \hline
-x_3-2j & x_3 + 2j + 1 & x_4 + 2j & -x_4 - 2j - 1 \\ \hline
y_3-2j & -y_3 + 2j + 1 & -y_4 + 2j & y_4 -2j - 1 \\ \hline
x_3-y_3 + 4j &  y_3 -x_3 - 4j -2 & y_4-x_4 - 4j & x_4 - y_4 + 4j + 2\\ \hline
\end{array}$
}\end{minipage}
\caption{\label {pairgroups} A $3 \times 8$ matrix of pairgroups (split over 2 lines)}
\end{figure}

Suppose that we have a set of 8 triples $(a,b,c)$ with $a+b=c$ that can  arranged in a {\em tile} so that the three row sums $[s_1,s_2,s_3]$ are the same for all the tiles (and necessarily the column sums are all zero).  We note here that much of the time it is possible to find tiles so that $s_1 =s_2 =s_3 =0$ which is certainly preferable.
\end{comment}

\medskip
We are now ready to describe our how our $3 \times 8$ tiles were found.

\begin{enumerate}
\item Given a Skolem sequence $S$ of order $n$, choose a subsequence $S'$ of the entries such that  each element appears either 0 or twice.
%\item Given $n$, find a partial Skolem sequence of order  $n/3$, that is, a partially filled sequence with each element appearing either 0 or twice such that $r(k)-\ell(k) = k$ for $1\leq k \leq n$,
\item Form the resulting sets with $a + b = c$ into sets of 8 triples each forming a tile (so these are $3\times 8$ tiles), where every element, up to sign, appears in at most one tile and where the three row sums $[s_1,s_2,s_3]$ are the same for all the tiles, preferably with  $s_1 =s_2 =s_3 =0$.  (We use specially-constructed Skolem sequences which make this step quite natural. See \cite{Tom-thesis} for details),
\item Call the remaining $3n_s$ elements {\em sporadic} -- these are the elements not appearing in $S'$.  For each of the sporadic elements place either $x$ or $-x$ in a $3 \times n_s$ {\em sporadic tile} with row sums determined so that the sum of rows of the sporadic tile plus all of the other tiles equals $0$ (again note that if each tile has all three row sums equal to $0$, then so will the sporadic tile),
\item Juxtapose the sporadic tile with the tiles found earlier to form the resulting matrix.
\end{enumerate}

Constructing the array in this manner makes it much easier to check it satisfies the Heffter properties.

It is possible to use $3 \times 4$ tiles, rather than $3 \times 8$ ones.  However, an apparent (uninvestigated) modular condition prevented us from constructing a family Heffter Arrays of size $4k+n_s$ for all $k>0$.  The ``Skolem-like'' sequences were constructed to contain four large families of entries and several sporadic entries; it happens that there is a particularly natural construction for $3 \times 8$ tiles using the entries from these large families (again -- much more details about this step can be found in \cite{Tom-thesis}.  The remaining sporadic entries, together with the entries of zero or one $3 \times 4$ tile, can be placed into a $3 \times (n_s+4\varepsilon)$ sporadic tile.  In the end, we were able to rearrange the entries using a single family of $3 \times 4$ tiles (as well as the sporadic tile), provided we could chose a sign for each tile, however in the constructions below we will present them as  $3 \times 8$ arrays.

We were able to keep the number of sporadic entries between 12 and 33, i.e., the sporadic tile was of size $3 \times r$ for $4 \le r <12$. Several congruences examples required small examples that did not fit the general pattern.  These were found by computer search, with the approach outlined below.
\begin{itemize}
 \item Compute a list of all triples $(a,b,a+b)$ where $a,b,a+b$ are positive representatives of entries in the sporadic tile.
 \item Use an EXACTCOVER solver to produce a random partition of our entries into triples $(a_i, b_i, a_i+b_i)$ for $1 \leq i \leq r$.
 \item Compute all vector sums
 \[
  \sum_{i=1}^{\lfloor r/2\rfloor} (x_i, y_i, z_i)
 \qquad
 \text{ and }
 \qquad
  \sum_{i=\lceil r/2\rceil}^{r} (x_i, y_i, z_i)
 \]
 where $\{x_i, y_i, z_i\} \in \{\{a_i,b_i,-a_i-b_i\},\{-a_i,-b_i,a_i+b_i\}\}$.
 \item Find a pair of vector sums $(x,y,z), (u,v,w)$ corresponding to $1 \leq i \leq \lfloor r/2\rfloor$ and $\lceil r/2\rceil \leq i \leq r$ respectively so that $(x+u,y+v,z+w)$ is the negation of the row sums of our $3 \times 8$ tiles.  If no such pair exists, select another partition from the EXACTCOVER solver.
\end{itemize}
These computations are performed in the ring $\ZZ[k]$ or when necessary, $\ZZ[k]/(6z+1)\ZZ$, with $z = 4k+n_s$ to ensure that the sporadic block works for all $k$.

The approach relies heavily on finding a Skolem sequence of order $n$. These exist if $n \equiv 0,1 \pmod{4}$,  exactly the cases for integer solutions. However in the other two cases we can use a related Skolem-type sequence.

A {\em $K$-near Skolem sequence} of order $n$ is a sequence $S = (s_1,s_2,\dots,s_{2n-2})$ of $n-1$ integers satisfying 1) every number $1 \le j \le n$, $j \notin K$ appears exactly twice, first in position $\ell(j)$ and second in position $r(j)$, and 2) for all $j\notin K$, $r(j) - \ell(j) = j$. An example of a $\{4\}$-near Skolem sequence of order 7 is $S_7^4 = (1,1,6,3,7,5,3,2,6,2,5,7)$. A $\{k\}$-near Skolem sequences of order $n$ exist  if and only if $n \equiv 0,1 \pmod{4}$ and $k$ is even, or $n \equiv 2,3 \pmod{4}$ and $k$ is odd \cite{S}.

To find non-integer $3 \times n$ arrays with $n \equiv 2,3 \pmod{4}$ we use a $\{2\}$-near Skolem sequence of order $n$, add in the column $(2,3n-1,3n)$, mark these three elements as sporadic, and continue as before.

To find integer $5 \times n$ arrays, we use $\br{1}$-near and $\br{1,2}$-near Skolem sequences to produce triples of the form $(j+1,\ell(j)+x,-\ell(j)-j-x)$ whose sum is $1$.  To these we append pairs $(y,-y-1)$, bringing the sum to zero.  Since $\br{1}$-near Skolem sequences exist for $n \equiv 1,2 \pmod{4}$, and we are able to cover the interval $[1,10n]$ with these columns.  For the other two cases, we use $\br{1,2}$-near Skolem sequences and the sporadic column $(1,-2,3,5n,5n-1)$ whose sum is $10n+1$.  Given these sets of columns, we find Heffter arrays similar to the $3 \times n$ cases; with a natural family of $5 \times 8$ blocks together with a brute-forced sporadic block.

All together, we need $\emptyset$-, $\{2\}$-, $\{1\}-$, and $\{1,2\}$-near Skolem sequences for all but a few small orders.  Constructions of these sequences were found through a combination of optimism, dedication, and computation.  The most dedicated reader could infer those constructions by reverse-engineering our Heffter arrays.  As these sequences are, in the end, tangential to our main result, we will not spend any additional time or space describing them or how we found them. Full details are availible in \cite{Tom-thesis}.

The remaining cases were found by examining all $2 \times 4$ blocks of the form 
\[
 \left[\begin{array}{rrrr}
  x+a & -x-c &  x+e & -x-g \\
 -x-b &  x+d & -x-f &  x+h 
 \end{array}\right]
\]
where $x$ is variable and $\{a,b,c,d,e,f,g,h\} = [1,8]$.  Certain of these blocks, together with the strip
\[
 \left[\begin{array}{rrrrr}
  x+1 & -x-2 &  -x-3 & x+4 \\
 \end{array}\right]
\]
are used to construct sets $7 \times 4$ and $9 \times 4$ tiles whose row and column sums are zero.  These tilesets provide enough flexibility to admit arbitrary-sized constructions.  Namely, we were able to find $7 \times 7$, $7 \times 9$, and $9 \times 9$ sporadic blocks whose row and column sums are zero (in modulus if necessary).

%%%%%%%%%%%%%%%%%%%%%%%%%%%%%%%%%%%%%%%%%%%%%%
\subsection {\label{justify}Automating the verification}
%%%%%%%%%%%%%%%%%%%%%%%%%%%%%%%%%%%%%%%%%%%%%%%%

In this section we will discuss the method which was used to verify the correctness of the constructions for all the cases.  Obviously -- first the constructions were found in a manner discussed in the prior section.  Then a Python program was written which applied the constructions and for every $m,n \geq 3$ output an $H(m,n)$ which was tested for correctness (program available on request).  Later, we repurposed that code to use Sage to generate symbolic expressions for the block constructions that it produced.   The Sage code was then enhanced to produce the \LaTeX \, source for the subsections which follow.  (We think that the output is rather nicer than one expects of a computer generated proof.)

Here, we describe how the code generates this output so that the reader may verify the result.  First, given a purported Heffter array $H=H(m,n)$ the code generates the tiles and computes their row and column sums as polynomials in several variables (modulo $2mn+1$ if and only if necessary).  Most of these tiles will  have row and column sums all zero.  Then, it examines the entries in each tile.
\begin{itemize}
 \item We assume the variables have a nonnegative value or else the constructions are invalid), and find their symbolic absolute value.
 \item For each entry $e$, we look for the entry $e+1$ in the tile, and if it is there, then $[e,e+1]$ is an interval covered by the tile.  Continue this process until all entries in the tile are used in an interval, and there are as few intervals as possible.
 \item Some tiles occur precisely once, and we're done with them at this point.  Other tiles have entries that are nonconstant polynomials in a variable $r$.  % We call these \emph{variable tiles}.
 \item Suppose a sequence of tiles $B_0, \cdots, B_f$ are to be made.  As we did with the individual entries, we assemble intervals $[a(r),b(r)]$ into long intervals $[a(0),b(f)]$ if $b(r)+1 = a(r+1)$ or $[a(f),b(0)]$ if $b(r)+1 = a(r-1)$.

 \begin{itemize}
  \item In a few cases, a few variable intervals don't tie together nicely.  We regroup these intervals into four sequences $(a_i(r), a_i(r)+4, a_i(r)+8, \cdots, b_i(r))$ for $i=0,1,2,3$.  These sequences interlace to produce a whole interval $[a_1(0),b_2(f)]$ together with two sporadic entries, $\{a_0(0)\}$ and $\{b_3(f)\}$.  These are stored with the collection of long intervals.
 \end{itemize}
\end{itemize}  

Some constructions then use a large shiftable Heffter array whose dimensions are both even.  We know these exist by Theorem \ref{evenbyeventhm} and the entries form a known interval.  We add that shifted interval to the collection.  Finally, we have some number of intervals $I_0 = [a_0, b_0], I_1 = [a_1, b_1], \cdots$.  Then, we assemble these long intervals: if $b_i + 1 = a_j$, then we write $[a_i, b_j] = I_i \cap I_j = I_i I_j$, where adjacency denotes interval concatenation.

In several of the $(3 \times n)$ cases, our constructions do not produce tiles whose row and column sums are zero  (this is possible, but we are using these constructions to maintain consistency with the arrays used in \cite{DM}).  In these cases, we show that the row sums of the entire array are zero.  In all other constructions, the tiles themselves have row and column sums zero (possibly over $\ZZ_{2mn+1}$), hence the entire array has row and column sum zero.

Since the proof is typeset by computer, we are able to retroactively relabel the intervals such that $[1,mn] = I_0 I_1 I_2 \cdots I_f$.  Putting all of this together, we see that each entry in the interval $[1, mn]$ is used precisely once and hence $sup(H)=[1, mn]$. Since the row and column sums are zero, we have indeed produced a Heffter array of the claimed dimensions.

%%%%
\subsection{ \label{3byother} Constructing  H(3,{\em n}) for all {\em n} $\geq$ 3} 
%%%%%

In Proposition \ref{3by8_theorem} we used tiles $A$ and $A_r$, $0 \le r \le k-1$ to construct an integer $3 \times n$ Heffter array for $n \equiv 0 \pmod 8$. In  Theorem \ref{3byn_theorem} below we will use similar techniques to find $H(3, n)$ for all $n\geq 3$.
The proof proceeds in 8 cases depending on the residue $n \pmod {8}$. There are also a few small cases to consider. The case $n \equiv 0 \pmod 8$ was given in Proposition \ref{3by8_theorem}, but will be included below with a (Sage) verification of included symbols. The proof of the remaining cases will be similar to that case -- we will concatenate the sporadic tile $A$ and the  tiles $A_{r}$ . We then give the Sage generated code that checks that the support of $H$ is $[1,mn]$ and that the row and column sums are all 0 over the integers or if necessary over $\mathbb Z_{6n+1}$.  An extensive discussion of how these tiles were found is in   \cite{Tom-thesis}.

\begin{theorem}\label{3byn_theorem} There exists a  Heffter array $H(3,n)$ for all $n \ge 3$. If $n \equiv 0,1 \pmod{4}$, then the array is integer. \end{theorem}

\begin{proof} 
The general constructions which follow cover all cases of $3 \times n$ Heffter arrays except the two small cases of $H(3,3)$ and $H(3,4)$.  These are presented in Example \ref{first_example}.

In each of the following cases we will write $n=8k+s$ for $s=7,8, \ldots 14$.  Our desired array $H$ will be built by concatenating $k+1$ tiles. Each begins with a $3 \times s$ sporadic tile $A$ followed by  the $3 \times 8$ tiles $A_r$, for $r = 0,\ldots,k-1$.  In each case we construct the  $3 \times n$ array as follows:
$$ H =\begin{array}{|c|c|c|c|c|c|} \hline 
A & A_0& A_1& \cdots &   A_{k-2}& A_{k-1} \\ \hline \end{array}.$$ 
\medskip

\noindent
{\bf n $\equiv$ 0 (mod 8)}. We write $n=8k+8$.  In this case

\[ A^T = \left[\begin{smallmatrix}
-12k - 13&4k + 4&8k + 9\\
-10k - 11&-8k - 7&18k + 18\\
4k + 6&18k + 17&-22k - 23\\
4k + 3&18k + 19&-22k - 22\\
10k + 10&4k + 5&-14k - 15\\
-4k - 8&-18k - 16&22k + 24\\
12k + 14&-2&-12k - 12\\
-1&-20k - 20&20k + 21\end{smallmatrix}\right] \mbox{ and }
A_r^T = \left[\begin{smallmatrix}
8k + 2r + 10&8k - 4r + 5&-16k + 2r - 15\\
-8k + 4r - 8&-16k - 2r - 16&24k - 2r + 24\\
14k - 2r + 14&-4k + 4r - 2&-10k - 2r - 12\\
-4k + 4r - 1&-18k - 2r - 20&22k - 2r + 21\\
-8k - 2r - 11&-8k + 4r - 3&16k - 2r + 14\\
8k - 4r + 6&16k + 2r + 17&-24k + 2r - 23\\
-14k + 2r - 13&4k - 4r&10k + 2r + 13\\
4k - 4r - 1&18k + 2r + 21&-22k + 2r - 20\end{smallmatrix}\right].
\]

The row sums of $A$ are $[4k,\, -2k,\, -2k]$ and the column sums of $A$ are 
$[0,\, 0,\, 0,\, 0,\, 0,\, 0,\, 0,\, 0]$. For each $0\leq r\leq k-1$ we see that the row sums of  
$A_r$ are $[-4,\, 2,\, 2]$ while the column sums are $[0,\, 0,\, 0,\, 0,\, 0,\, 0,\, 0,\, 0]$.
So each of these tiles are column Heffter but not row Heffter. Now concatenate these tiles into a $3 \times n$ array $H$ as described above.

We claim that $H$ is an integer Heffter array $H(3,n)$. First, since each tile is column Heffter, so is $H$. Second, the sum of the rows of the $A_r$'s is $k \times [-4,\, 2,\, 2] =[-4k,\, 2k,\, 2k] $ which when added to the row sums in $A$ gives that every row sum is indeed 0, as desired.

Now we consider the support of $H$.
The entries in $A$ cover the
intervals $I_{0 } = [1,2]$, $I_{2 } = [4k + 3,4k + 6]$, $I_{8 } = \{8k + 9\}$, $I_{4 } = \{4k + 8\}$, $I_{6 } = \{8k + 7\}$, $I_{10 } = [10k + 10,10k + 11]$, $I_{14 } = \{14k + 15\}$, $I_{21 } = [22k + 22,22k + 24]$, $I_{12 } = [12k + 12,12k + 14]$, $I_{19 } = [20k + 20,20k + 21]$ and $I_{17 } = [18k + 16,18k + 19]$.

The entries in $A_r$ cover the
intervals $\{8k - 4r + 8\}$, $[8k - 4r + 5,8k - 4r + 6]$, $[4k - 4r - 1,4k - 4r + 2]$, $\{8k - 4r + 3\}$, $[8k + 2r + 10,8k + 2r + 11]$, $[10k + 2r + 12,10k + 2r + 13]$, $[18k + 2r + 20,18k + 2r + 21]$, $[16k - 2r + 14,16k - 2r + 15]$, $[16k + 2r + 16,16k + 2r + 17]$, $[24k - 2r + 23,24k - 2r + 24]$, $[14k - 2r + 13,14k - 2r + 14]$ and $[22k - 2r + 20,22k - 2r + 21]$.
Considering $0 \leq r \leq k - 1$, these tiles cover the 
intervals $I_{11 } = [3,4k + 2]$, $I_{15 } = [8k + 10,10k + 9]$, $I_{16 } = [10k + 12,12k + 11]$, $I_{22 } = [18k + 20,20k + 19]$, $I_{13 } = [14k + 16,16k + 15]$, $I_{20 } = [16k + 16,18k + 15]$, $I_{3 } = [22k + 25,24k + 24]$, $I_{5 } = [12k + 15,14k + 14]$ and $I_{7 } = [20k + 22,22k + 21]$.

Additionally, we split the intervals $\{8k - 4r + 8\}$, $[8k - 4r + 5,8k - 4r + 6]$ and $\{8k - 4r + 3\}$ into the sequences 
$ (4k + 7, 4k + 11, \cdots, 8k + 3), 
 (4k + 9, 4k + 13, \cdots, 8k + 5), 
 (4k + 10, 4k + 14, \cdots, 8k + 6), $
and
$(4k + 12, 4k + 16, \cdots, 8k + 8), $
 and rejoin them into the intervals $I_{3 } = \{4k + 7\}$, $I_{5 } = [4k + 9,8k + 6]$ and $I_{7 } = \{8k + 8\}$.

Concatenating these intervals, we have that $sup(H)= [1,24k + 24] = I_0 I_1 \cdots I_{22}$
and hence have constructed an integer $(3 \times 8k+8)$ Heffter array for all $k\geq 0$.

\bigskip
\noindent
{\bf n $\equiv$ 1 (mod 8)}. 
We write $n=8k+9$. In this case

\[ A^T = \left[\begin{smallmatrix}
8k + 7&8k + 10&-16k - 17\\
10k + 12&8k + 9&-18k - 21\\
16k + 18&-12k - 14&-4k - 4\\
4k + 6&-22k - 26&18k + 20\\
4k + 3&18k + 22&-22k - 25\\
-4k - 5&14k + 16&-10k - 11\\
-12k - 13&-2&12k + 15\\
-22k - 27&4k + 8&18k + 19\\
-1&-20k - 23&20k + 24\end{smallmatrix}\right] \mbox{ and }
 A_r^T = \left[\begin{smallmatrix}
-8k + 4r - 5&16k - 2r + 16&-8k - 2r - 11\\
-10k - 2r - 13&-4k + 4r - 2&14k - 2r + 15\\
-24k + 2r - 27&8k - 4r + 8&16k + 2r + 19\\
-4k + 4r - 1&-18k - 2r - 23&22k - 2r + 24\\
8k - 4r + 3&-16k + 2r - 15&8k + 2r + 12\\
10k + 2r + 14&4k - 4r&-14k + 2r - 14\\
24k - 2r + 26&-8k + 4r - 6&-16k - 2r - 20\\
4k - 4r - 1&18k + 2r + 24&-22k + 2r - 23\end{smallmatrix}\right]. \]

The row sums  of $A$ are $[4k,\, -2k,\, -2k]$ and the column sums are $[0,\, 0,\, 0,\, 0,\, 0,\, 0,\, 0,\, 0,\, 0]$. 
For each $0\leq r\leq k-1$ we see that the row sums of  
$A_r$ are $[-4,\, 2,\, 2]$ while the column sums are $[0,\, 0,\, 0,\, 0,\, 0,\, 0,\, 0,\, 0]$.
So each of these tiles are column Heffter but not row Heffter. Again form $H$ as described above.
 It is easy now to see that all the row and column sums are zero.

 The entries in $A$ cover the
intervals $I_{0 } = [1,2]$, $I_{2 } = [4k + 3,4k + 6]$, $I_{6 } = \{8k + 7\}$, $I_{4 } = \{4k + 8\}$, $I_{8 } = [8k + 9,8k + 10]$, $I_{16 } = [16k + 17,16k + 18]$, $I_{14 } = \{14k + 16\}$, $I_{22 } = [22k + 25,22k + 27]$, $I_{10 } = [10k + 11,10k + 12]$, $I_{20 } = [20k + 23,20k + 24]$, $I_{18 } = [18k + 19,18k + 22]$ and $I_{12 } = [12k + 13,12k + 15]$.

 The entries in $A_r$ cover the
intervals $\{8k - 4r + 8\}$, $[8k - 4r + 5,8k - 4r + 6]$, $[4k - 4r - 1,4k - 4r + 2]$, $\{8k - 4r + 3\}$, $[8k + 2r + 11,8k + 2r + 12]$, $[18k + 2r + 23,18k + 2r + 24]$, $[14k - 2r + 14,14k - 2r + 15]$, $[10k + 2r + 13,10k + 2r + 14]$, $[16k + 2r + 19,16k + 2r + 20]$, $[24k - 2r + 26,24k - 2r + 27]$, $[16k - 2r + 15,16k - 2r + 16]$ and $[22k - 2r + 23,22k - 2r + 24]$.
Considering $0 \leq r \leq k - 1$, these tiles cover the 
intervals $I_{19 } = [3,4k + 2]$, $I_{11 } = [8k + 11,10k + 10]$, $I_{17 } = [18k + 23,20k + 22]$, $I_{23 } = [12k + 16,14k + 15]$, $I_{15 } = [10k + 13,12k + 12]$, $I_{21 } = [16k + 19,18k + 18]$, $I_{3 } = [22k + 28,24k + 27]$, $I_{5 } = [14k + 17,16k + 16]$ and $I_{7 } = [20k + 25,22k + 24]$.
Additionally, we split the intervals $\{8k - 4r + 8\}$, $[8k - 4r + 5,8k - 4r + 6]$ and $\{8k - 4r + 3\}$ into the sequences 
$(4k + 7, 4k + 11, \cdots, 8k + 3), 
 (4k + 9, 4k + 13, \cdots, 8k + 5), 
 (4k + 10, 4k + 14, \cdots, 8k + 6)$, 
and
$(4k + 12, 4k + 16, \cdots, 8k + 8)$, 
 and rejoin them into the intervals $I_{3 } = \{4k + 7\}$, $I_{5 } = [4k + 9,8k + 6]$ and $I_{7 } = \{8k + 8\}$.

Concatenating these intervals, we have covered $[1,24k + 27] = I_0 I_1 \cdots I_{23}$ as desired and thus have constructed an integer $(3 \times 8k+9)$ Heffter array for all $k \geq 0$.

\bigskip \noindent
{\bf n $\equiv$ 2 (mod 8)}. 
We write $n=8k+10$. In this case
\[ A^T = \left[\begin{smallmatrix}
24k + 30&24k + 29&2\\
16k + 21&-8k - 11&-8k - 10\\
10k + 13&-10k - 14&1\\
8k + 8&12k + 16&-20k - 24\\
4k + 5&16k + 20&-20k - 25\\
8k + 9&12k + 17&-20k - 26\\
-4k - 7&14k + 19&-10k - 12\\
12k + 15&3&-12k - 18\\
-20k - 27&4&20k + 23\\
-4k - 6&-18k - 22&22k + 28\end{smallmatrix}\right] \mbox{ and }
 A_r^T = \left[\begin{smallmatrix}
-8k + 4r - 7&16k - 2r + 19&-8k - 2r - 12\\
10k + 2r + 15&4k - 4r + 3&-14k + 2r - 18\\
-22k + 2r - 27&4k - 4r + 4&18k + 2r + 23\\
-8k + 4r - 6&-16k - 2r - 22&24k - 2r + 28\\
8k - 4r + 5&-16k + 2r - 18&8k + 2r + 13\\
-10k - 2r - 16&-4k + 4r - 1&14k - 2r + 17\\
22k - 2r + 26&-4k + 4r - 2&-18k - 2r - 24\\
8k - 4r + 4&16k + 2r + 23&-24k + 2r - 27\end{smallmatrix}\right].
\]

The row sums of $A$ are $[54k + 61,\, 42k + 61,\, -48k - 61]$ and the column sums are $[48k + 61,\, 0,\, 0,\, 0,\, 0,\, 0,\, 0,\, 0,\, 0,\, 0]$.  For each $0\leq r\leq k-1$ we see that the row sums of  
$A_r$ are $[-6,\, 6,\, 0]$ while the column sums are $[0,\, 0,\, 0,\, 0,\, 0,\, 0,\, 0,\, 0]$.
If we form $H$ as described above we see that all the row and column sums are zero modulo $48k+61 = 6n+1$.

 The entries in $A$ cover the
intervals $I_{0 } = [1,4]$, $I_{2 } = [4k + 5,4k + 7]$, $I_{4 } = [8k + 8,8k + 11]$, $I_{16 } = [20k + 23,20k + 27]$, $I_{6 } = [10k + 12,10k + 14]$, $I_{18 } = \{22k + 28\}$, $I_{14 } = \{18k + 22\}$, $I_{10 } = \{14k + 19\}$, $I_{8 } = [12k + 15,12k + 18]$, $I_{20 } = [24k + 29,24k + 30]$ and $I_{12 } = [16k + 20,16k + 21]$.

 The entries in $A_r$ cover the
intervals $[4k - 4r + 1,4k - 4r + 4]$, $[8k - 4r + 4,8k - 4r + 7]$, $[8k + 2r + 12,8k + 2r + 13]$, $[18k + 2r + 23,18k + 2r + 24]$, $[22k - 2r + 26,22k - 2r + 27]$, $[14k - 2r + 17,14k - 2r + 18]$, $[16k - 2r + 18,16k - 2r + 19]$, $[10k + 2r + 15,10k + 2r + 16]$, $[16k + 2r + 22,16k + 2r + 23]$ and $[24k - 2r + 27,24k - 2r + 28]$.
Considering $0 \leq r \leq k - 1$, these tiles cover the 
intervals $I_{1 } = [5,4k + 4]$, $I_{3 } = [4k + 8,8k + 7]$, $I_{5 } = [8k + 12,10k + 11]$, $I_{15 } = [18k + 23,20k + 22]$, $I_{17 } = [20k + 28,22k + 27]$, $I_{9 } = [12k + 19,14k + 18]$, $I_{11 } = [14k + 20,16k + 19]$, $I_{7 } = [10k + 15,12k + 14]$, $I_{13 } = [16k + 22,18k + 21]$ and $I_{19 } = [22k + 29,24k + 28]$.

Concatenating these intervals, we have covered $[1,24k + 30] = I_0 I_1 \cdots I_{20}$ as desired and thus have constructed a $(3 \times 8k+10)$ Heffter array for all $k \geq 0$.

\bigskip\noindent
{\bf n $\equiv$ 3 (mod 8)}. 
We write $n=8k+11$. In this case
\[ A^T = \left[\begin{smallmatrix}
24k + 33&24k + 32&2\\
8k + 11&-16k - 23&8k + 12\\
8k + 13&-12k - 18&4k + 5\\
4k + 6&10k + 15&-14k - 21\\
1&20k + 27&-20k - 28\\
-12k - 17&-8k - 9&20k + 26\\
8k + 10&14k + 20&-22k - 30\\
-14k - 22&4k + 8&10k + 14\\
22k + 31&-4k - 7&-18k - 24\\
4&-20k - 29&20k + 25\\
-3&-12k - 16&12k + 19\end{smallmatrix}\right]
\mbox{ and }
A_r^T = \left[\begin{smallmatrix}
-16k + 2r - 22&8k - 4r + 8&8k + 2r + 14\\
24k - 2r + 31&-8k + 4r - 7&-16k - 2r - 24\\
4k - 4r + 4&-22k + 2r - 29&18k + 2r + 25\\
-4k + 4r - 3&-10k - 2r - 16&14k - 2r + 19\\
16k - 2r + 21&-8k + 4r - 6&-8k - 2r - 15\\
-24k + 2r - 30&8k - 4r + 5&16k + 2r + 25\\
-4k + 4r - 2&22k - 2r + 28&-18k - 2r - 26\\
4k - 4r + 1&10k + 2r + 17&-14k + 2r - 18\end{smallmatrix}\right].
\]

The row sums of $A$ are $[48k + 67,\, 0,\, 0]$ and the column sums are $[48k + 67,\, 0,\, 0,\, 0,\, 0,\, 0,\, 0,\, 0,\,$ $ 0,\, 0,\, 0]$.  For each $0\leq r\leq k-1$ we see that all of the row  and column sums of
$A_r$ are zero.
If we form $H$ as described above we see that all the row and column sums are zero modulo $48k+67 = 6n+1$.

The entries in $A$ cover the
intervals $I_{0 } = [1,4]$, $I_{2 } = [4k + 5,4k + 8]$, $I_{4 } = [8k + 9,8k + 13]$, $I_{20 } = [24k + 32,24k + 33]$, $I_{10 } = [14k + 20,14k + 22]$, $I_{16 } = [20k + 25,20k + 29]$, $I_{18 } = [22k + 30,22k + 31]$, $I_{8 } = [12k + 16,12k + 19]$, $I_{12 } = \{16k + 23\}$, $I_{6 } = [10k + 14,10k + 15]$ and $I_{14 } = \{18k + 24\}$.

 The entries in $A_r$ cover the
intervals $[4k - 4r + 1,4k - 4r + 4]$, $[8k - 4r + 5,8k - 4r + 8]$, $[8k + 2r + 14,8k + 2r + 15]$, $[14k - 2r + 18,14k - 2r + 19]$, $[22k - 2r + 28,22k - 2r + 29]$, $[18k + 2r + 25,18k + 2r + 26]$, $[16k + 2r + 24,16k + 2r + 25]$, $[10k + 2r + 16,10k + 2r + 17]$, $[16k - 2r + 21,16k - 2r + 22]$ and $[24k - 2r + 30,24k - 2r + 31]$.
Considering $0 \leq r \leq k - 1$, these tiles cover the 
intervals $I_{1 } = [5,4k + 4]$, $I_{3 } = [4k + 9,8k + 8]$, $I_{5 } = [8k + 14,10k + 13]$, $I_{9 } = [12k + 20,14k + 19]$, $I_{17 } = [20k + 30,22k + 29]$, $I_{15 } = [18k + 25,20k + 24]$, $I_{13 } = [16k + 24,18k + 23]$, $I_{7 } = [10k + 16,12k + 15]$, $I_{11 } = [14k + 23,16k + 22]$ and $I_{19 } = [22k + 32,24k + 31]$.

Concatenating these intervals, we have covered $[1,24k + 33] = I_0 I_1 \cdots I_{20}$ and thus have constructed a $(3 \times 8k+11)$ Heffter array for all $k \geq 0$.

\bigskip\noindent
{\bf n $\equiv$ 4 (mod 8)}. We write $n=8k+12$. In this case
\[ A^T = \left[\begin{smallmatrix}
8k + 13&4k + 6&-12k - 19\\
10k + 16&8k + 11&-18k - 27\\
22k + 34&-4k - 8&-18k - 26\\
-4k - 5&22k + 33&-18k - 28\\
4k + 7&-14k - 22&10k + 15\\
-22k - 35&4k + 10&18k + 25\\
-12k - 18&-2&12k + 20\\
-1&-20k - 30&20k + 31\\
-14k - 23&10k + 14&4k + 9\\
-4k - 12&-18k - 24&22k + 36\\
12k + 21&-12k - 17&-4\\
3&20k + 29&-20k - 32\end{smallmatrix}\right]
\mbox{ and }
A_r^T= \left[\begin{smallmatrix}
-16k + 2r - 23&8k + 2r + 14&8k - 4r + 9\\
-8k + 4r - 12&-16k - 2r - 24&24k - 2r + 36\\
14k - 2r + 21&-10k - 2r - 17&-4k + 4r - 4\\
4k - 4r + 3&18k + 2r + 29&-22k + 2r - 32\\
16k - 2r + 22&-8k - 2r - 15&-8k + 4r - 7\\
8k - 4r + 10&16k + 2r + 25&-24k + 2r - 35\\
-14k + 2r - 20&10k + 2r + 18&4k - 4r + 2\\
-4k + 4r - 1&-18k - 2r - 30&22k - 2r + 31\end{smallmatrix}\right].
 \]

The row and column sums of $A$ and each of the $A_r$ are all zero, hence it is trivial to see that all row and column sums of $H$ are also zero.

 The entries in $A$ cover the
intervals $I_{0 } = [1,4]$, $I_{2 } = [4k + 5,4k + 10]$, $I_{8 } = \{8k + 13\}$, $I_{4 } = \{4k + 12\}$, $I_{6 } = \{8k + 11\}$, $I_{21 } = [22k + 33,22k + 36]$, $I_{14 } = [14k + 22,14k + 23]$, $I_{12 } = [12k + 17,12k + 21]$, $I_{17 } = [18k + 24,18k + 28]$, $I_{19 } = [20k + 29,20k + 32]$ and $I_{10 } = [10k + 14,10k + 16]$.

The entries in $A_r$ cover the
intervals $[4k - 4r + 1,4k - 4r + 4]$, $\{8k - 4r + 7\}$, $[8k - 4r + 9,8k - 4r + 10]$, $[8k + 2r + 14,8k + 2r + 15]$, $\{8k - 4r + 12\}$, $[18k + 2r + 29,18k + 2r + 30]$, $[14k - 2r + 20,14k - 2r + 21]$, $[24k - 2r + 35,24k - 2r + 36]$, $[22k - 2r + 31,22k - 2r + 32]$, $[10k + 2r + 17,10k + 2r + 18]$, $[16k - 2r + 22,16k - 2r + 23]$ and $[16k + 2r + 24,16k + 2r + 25]$.
Considering $0 \leq r \leq k - 1$, these tiles cover the 
intervals $I_{1 } = [5,4k + 4]$, $I_{13 } = [8k +14,10k +13]$, $I_{20 } = [18k + 29,20k + 28]$, $I_{11 } = [12k + 22,14k + 21]$, $I_{15 } = [22k + 37,24k + 36]$, $I_{16 } = [20k + 33,22k + 32]$, $I_{3 } = [10k + 17,12k + 16]$, $I_{5 } = [14k + 24,16k + 23]$ and $I_{7 } = [16k + 24,18k + 23]$.
Additionally, we split the intervals $\{8k - 4r + 7\}$, $[8k - 4r + 9,8k - 4r + 10]$ and $\{8k - 4r + 12\}$ into the sequences 
$ (4k + 11, 4k + 15, \cdots, 8k + 7), 
 (4k + 13, 4k + 17, \cdots, 8k + 9), 
 (4k + 14, 4k + 18, \cdots, 8k + 10)$ 
and $ (4k + 16, 4k + 20, \cdots, 8k + 12),$
 and rejoin them into the intervals $I_{3 } = \{4k + 11\}$, $I_{5 } = [4k + 13,8k + 10]$ and $I_{7 } = \{8k + 12\}$.

Concatenating these intervals, we have covered $[1,24k + 36] = I_0 I_1 \cdots I_{22}$ and thus have constructed an integer $(3 \times 8k+12)$ Heffter array for all $k \geq 0$.

\bigskip\noindent
{\bf n $\equiv$ 5 (mod 8)}. We write $n=8k+5$. Let
\[ A^T = \left[\begin{smallmatrix}
8k + 6&-16k - 9&8k + 3\\
10k + 7&8k + 5&-18k - 12\\
-16k - 10&4k + 2&12k + 8\\
-4k - 4&-18k - 11&22k + 15\\
4k + 1&18k + 13&-22k - 14\end{smallmatrix}\right] \mbox{ and }
A_r^T = \left[\begin{smallmatrix}
-8k + 4r - 1&16k - 2r + 8&-8k - 2r - 7\\
-14k + 2r - 8&4k - 4r&10k + 2r + 8\\
16k + 2r + 11&8k - 4r + 4&-24k + 2r - 15\\
4k - 4r - 1&18k + 2r + 14&-22k + 2r - 13\\
8k - 4r - 1&-16k + 2r - 7&8k + 2r + 8\\
14k - 2r + 7&-4k + 4r + 2&-10k - 2r - 9\\
-16k - 2r - 12&-8k + 4r - 2&24k - 2r + 14\\
-4k + 4r + 3&-18k - 2r - 15&22k - 2r + 12\end{smallmatrix}\right].
\]

Then $A$ has all column sums equaling zero and has row sums $[2k,\, -4k,\, 2k]$. For each $0\leq r\leq k-1$ we see that the row sums of  
$A_r$ are $[-2,\, 4,\, -2]$ while the column sums are $[0,\, 0,\, 0,\, 0,\, 0,\, 0,\, 0,\, 0]$.
If we form $H$ as described above it is easy now to see that all the row and column sums are zero.

The entries in $A$ cover the
intervals $I_{3 } = \{4k + 4\}$, $I_{7 } = [8k + 5,8k + 6]$, $I_{5 } = \{8k + 3\}$, $I_{1 } = [4k + 1,4k + 2]$, $I_{9 } = \{10k + 7\}$, $I_{11 } = \{12k + 8\}$, $I_{14 } = [16k + 9,16k + 10]$, $I_{19 } = [22k + 14,22k + 15]$ and $I_{16 } = [18k + 11,18k + 13]$.

 The entries in $A_r$ cover the
intervals $[4k - 4r - 3,4k - 4r]$, $\{8k - 4r - 1\}$, $\{8k - 4r + 4\}$, $[8k + 2r + 7,8k + 2r + 8]$, $[8k - 4r + 1,8k - 4r + 2]$, $[10k + 2r + 8,10k + 2r + 9]$, $[16k - 2r + 7,16k - 2r + 8]$, $[14k - 2r + 7,14k - 2r + 8]$, $[22k - 2r + 12,22k - 2r + 13]$, $[24k - 2r + 14,24k - 2r + 15]$, $[16k + 2r + 11,16k + 2r + 12]$ and $[18k + 2r + 14,18k + 2r + 15]$.
Considering $0 \leq r \leq k - 1$, these tiles cover the 
intervals $I_{0 } = [1,4k]$, $I_{13 } = [8k + 7,10k + 6]$, $I_{18 } = [10k + 8,12k + 7]$, $I_{20 } = [14k + 9,16k + 8]$, $I_{15 } = [12k + 9,14k + 8]$, $I_{17 } = [20k + 14,22k + 13]$, $I_{2 } = [22k + 16,24k + 15]$, $I_{4 } = [16k + 11,18k + 10]$ and $I_{6 } = [18k + 14,20k + 13]$.
Additionally, we split the intervals $\{8k - 4r - 1\}$, $\{8k - 4r + 4\}$ and $[8k - 4r + 1,8k - 4r + 2]$ into the sequences 
$ (4k + 3, 4k + 7, \cdots, 8k - 1), 
 (4k + 5, 4k + 9, \cdots, 8k + 1), 
(4k + 6, 4k + 10, \cdots, 8k + 2),$
and
$ (4k + 8, 4k + 12, \cdots, 8k + 4),$
 and rejoin them into the intervals $I_{2 } = \{4k + 3\}$, $I_{4 } = [4k + 5,8k + 2]$ and $I_{6 } = \{8k + 4\}$.

Concatenating these intervals, we have covered $[1,24k + 15] = I_0 I_1 \cdots I_{20}$
and thus have constructed an integer $(3 \times 8k+5)$ Heffter array for all $k \geq 0$.

\bigskip\noindent
{\bf n $\equiv$ 6 (mod 8)}. We write $n=8k+6$. Let
\[ A = \left[\begin{smallmatrix}
24k + 18&2&24k + 17\\
-16k - 13&8k + 6&8k + 7\\
-1&-10k - 8&10k + 9\\
8k + 4&-20k - 14&12k + 10\\
-4k - 3&-16k - 12&20k + 15\\
-8k - 5&-12k - 11&20k + 16\end{smallmatrix}\right] \mbox { and }
A_r = \left[\begin{smallmatrix}
-8k + 4r - 3&16k - 2r + 11&-8k - 2r - 8\\
-4k + 4r - 1&-10k - 2r - 10&14k - 2r + 11\\
-4k + 4r - 2&22k - 2r + 16&-18k - 2r - 14\\
8k - 4r + 2&16k + 2r + 14&-24k + 2r - 16\\
8k - 4r + 1&-16k + 2r - 10&8k + 2r + 9\\
4k - 4r - 1&10k + 2r + 11&-14k + 2r - 10\\
4k - 4r&-22k + 2r - 15&18k + 2r + 15\\
-8k + 4r&-16k - 2r - 15&24k - 2r + 15\end{smallmatrix}\right].
 \]

Then $A$ has column sums $[48k + 37,\, 0,\, 0,\, 0,\, 0,\, 0]$ and row sums $[4k,\, -50k - 37,\, 94k + 74].$ For each $0\leq r\leq k-1$ we see that the row sums of  
$A_r$ are $[-4,\, 2,\, 2]$ while the column sums are $[0,\, 0,\, 0,\, 0,\, 0,\, 0,\, 0,\, 0]$.
If we form $H$ as described above it is easy now to see that all  column sums are zero and the row sums are $[0, -48k-37, 96k+74]$, hence all row sums are $0$ modulo $6n+1$.

The entries in $A$ cover the
intervals $I_{0 } = [1,2]$, $I_{4 } = [8k + 4,8k + 7]$, $I_{2 } = \{4k + 3\}$, $I_{6 } = [10k + 8,10k + 9]$, $I_{8 } = [12k + 10,12k + 11]$, $I_{11 } = [16k + 12,16k + 13]$, $I_{17 } = [24k + 17,24k + 18]$ and $I_{14 } = [20k + 14,20k + 16]$.

 The entries in $A_r$ cover the
intervals $[8k - 4r,8k - 4r + 3]$, $[4k - 4r - 1,4k - 4r + 2]$, $[8k + 2r + 8,8k + 2r + 9]$, $[16k + 2r + 14,16k + 2r + 15]$, $[22k - 2r + 15,22k - 2r + 16]$, $[10k + 2r + 10,10k + 2r + 11]$, $[24k - 2r + 15,24k - 2r + 16]$, $[18k + 2r + 14,18k + 2r + 15]$, $[16k - 2r + 10,16k - 2r + 11]$ and $[14k - 2r + 10,14k - 2r + 11]$.
Considering $0 \leq r \leq k - 1$, these tiles cover the 
intervals $I_{3 } = [4k + 4,8k + 3]$, $I_{1 } = [3,4k + 2]$, $I_{5 } = [8k + 8,10k + 7]$, $I_{12 } = [16k + 14,18k + 13]$, $I_{15 } = [20k + 17,22k + 16]$, $I_{7 } = [10k + 10,12k + 9]$, $I_{16 } = [22k + 17,24k + 16]$, $I_{13 } = [18k + 14,20k + 13]$, $I_{10 } = [14k + 12,16k + 11]$ and $I_{9 } = [12k + 12,14k + 11]$.

Concatenating these intervals, we have covered $[1,24k + 18] = I_0 I_1 \cdots I_{17}$ and hence have constructed a 
 $(3 \times 8k+6)$ Heffter array  for all $k \geq 0$.

\bigskip\noindent
{\bf n $\equiv$ 7 (mod 8)}. We write $n=8k+7$. Let

\[ A = \left[\begin{smallmatrix}
24k + 21&2&24k + 20\\
16k + 15&-8k - 8&-8k - 7\\
4k + 3&-12k - 12&8k + 9\\
-4k - 4&14k + 14&-10k - 10\\
-20k - 18&1&20k + 17\\
-12k - 11&20k + 16&-8k - 5\\
-8k - 6&-14k - 13&22k + 19\end{smallmatrix}\right]
\mbox{ and }
A_r = \left[\begin{smallmatrix}
-16k + 2r - 14&8k + 2r + 10&8k - 4r + 4\\
-8k + 4r - 3&-16k - 2r - 16&24k - 2r + 19\\
-18k - 2r - 16&22k - 2r + 18&-4k + 4r - 2\\
4k - 4r + 1&10k + 2r + 11&-14k + 2r - 12\\
16k - 2r + 13&-8k - 2r - 11&-8k + 4r - 2\\
8k - 4r + 1&16k + 2r + 17&-24k + 2r - 18\\
18k + 2r + 17&-22k + 2r - 17&4k - 4r\\
-4k + 4r + 1&-10k - 2r - 12&14k - 2r + 11\end{smallmatrix}\right].
\]

Now the column sums of $A$ are $[48k + 43,\, 0,\, 0,\, 0,\, 0,\, 0,\, 0]$ while the row sums are $[0,\, 0,\, 48k + 43]$. For each $0\leq r\leq k-1$ we see that all the row  and column sums of  
$A_r$ all zero.
If we form $H$ as described above it is easy now to see that all  row and column sums of $H$ are  $0$ modulo $6n+1$.

 The entries in $A$ cover the
intervals $I_{0 } = [1,2]$, $I_{2 } = [4k + 3,4k + 4]$, $I_{4 } = [8k + 5,8k + 9]$, $I_{15 } = [20k + 16,20k + 18]$, $I_{6 } = \{10k + 10\}$, $I_{12 } = \{16k + 15\}$, $I_{19 } = [24k + 20,24k + 21]$, $I_{17 } = \{22k + 19\}$, $I_{8 } = [12k + 11,12k + 12]$ and $I_{10 } = [14k + 13,14k + 14]$.

 The entries in $A_r$ cover the
intervals $[8k - 4r + 1,8k - 4r + 4]$, $[4k - 4r - 1,4k - 4r + 2]$, $[8k + 2r + 10,8k + 2r + 11]$, $[16k + 2r + 16,16k + 2r + 17]$, $[22k - 2r + 17,22k - 2r + 18]$, $[10k + 2r + 11,10k + 2r + 12]$, $[16k - 2r + 13,16k - 2r + 14]$, $[24k - 2r + 18,24k - 2r + 19]$, $[14k - 2r + 11,14k - 2r + 12]$ and $[18k + 2r + 16,18k + 2r + 17]$.
Considering $0 \leq r \leq k - 1$, these tiles cover the 
intervals $I_{3 } = [4k + 5,8k + 4]$, $I_{1 } = [3,4k + 2]$, $I_{5 } = [8k + 10,10k + 9]$, $I_{13 } = [16k + 16,18k + 15]$, $I_{16 } = [20k + 19,22k + 18]$, $I_{7 } = [10k + 11,12k + 10]$, $I_{11 } = [14k + 15,16k + 14]$, $I_{18 } = [22k + 20,24k + 19]$, $I_{9 } = [12k + 13,14k + 12]$ and $I_{14 } = [18k + 16,20k + 15]$.

Concatenating these intervals, we have covered $[1,24k + 21] = I_0 I_1 \cdots I_{19}$ and hence $H$ is  a $(3 \times 8k+7)$ Heffter array
for all $k \geq 0$.  This concludes the proof of the theorem.
\end{proof}

%%%%%%%%%%%%%%%%%%%%%%%%%%%%%%%%%%%%%%%%%%%%%%
\section {\label{5byn} Constructing  H(5,$ n$) for all $n \geq$ 3}
%%%%%%%%%%%%%%%%%%%%%%%%%%%%%%%%%%%%%%%%%%%%%%

In Section \ref{3byn} we gave $3 \times n$ Heffter arrays for all $n$ and integer Heffter arrays if $n \equiv 0,1 \pmod 4$.  We will use these to construct $m \times n$ arrays by induction on $m$ in Section \ref{extending}. In general the induction step is of size 4 so these will be induction steps in the case when $m \equiv 3 \pmod 4$. To cover the case $m \equiv 1 \pmod 4$ we will construct $5 \times n$ Heffter arrays for all values of $n$. We do this in the next theorem.

\begin{theorem}\label{5byn_theorem} There exists a  Heffter array $H(5,n)$ for all $n \ge 3$. If $n \equiv 0,3 \pmod{4}$, then the array is integer. \end{theorem}

\begin{proof}

The general constructions which follow cover all cases of  $H(5,n)$ arrays except the four  small cases of $H(5,n)$ with $n = 3,4,5,6$.  An $H(5,3)$ is constructed in Theorem \ref{3byn_theorem}, below we give an integer $H(5,4)$,  an $H(5,5)$ and an $H(5,6)$ .
\renewcommand{\arraycolsep}{3pt}
{\small$$
\begin{array}{|c|c|c|c|} \hline
7&-16&-10&19 \\ \hline
-12&15&17&-20\\ \hline
-2&9&-18&11\\ \hline
6&5&3&-14\\ \hline
1&-13&8&4\\ \hline
\multicolumn{4}{c}{H(4,5)}
\end{array}  \hspace{.5in}
 \begin{array}{|c|c|c|c|c|} \hline
1&5&6&7&-19\\ \hline
2&8&12&15&14\\ \hline
3&9&-21&22&-13\\ \hline
4&11&-25&-24&-17\\ \hline
-10&18&-23&-20&-16\\ \hline
\multicolumn{5}{c}{H(5,5)}
\end{array} \hspace{.5in}
 \begin{array}{|c|c|c|c|c|c|} \hline
1&-8&-7&15&26&-27\\ \hline
-2&20&-11&24&-25&-6\\ \hline
29&-19&17&-4&-10&-13\\ \hline
30&-9&-21&-23&-5&28\\ \hline
3&16&22&-12&14&18\\ \hline
\multicolumn{6}{c}{H(5,6)}
\end{array}
$$
}\renewcommand{\arraycolsep}{6pt}

The  constructions here are  quite similar to that of the $3 \times n$ case shown in Theorem \ref{3byn_theorem}.  In each of the following cases we will write $n=8k+s$ for $s=7,8, \ldots 14$.  Our desired array $H$ will be built by concatenating $k+1$ tiles. Each begins with a $5 \times s$ sporadic tile $A$ followed by  the $5 \times 8$ tiles $A_r$, for $r = 0,\ldots,k-1$.  In each case we construct the  $5 \times n$ array as follows:
$$ H =\begin{array}{|c|c|c|c|c|c|} \hline 
A & A_0& A_1& \cdots &   A_{k-2}& A_{k-1} \\ \hline \end{array}.$$ 
\medskip

\bigskip\noindent
{\bf n $\equiv$ 0 (mod 8)}. We write $n=8k+8$. Let

\[ A^T=\left[\begin{smallmatrix}
8k + 10&8k + 11&1&-2&-16k - 20\\
-4k - 5&-8k - 12&24k + 28&12k + 16&-24k - 27\\
-12k - 15&24k + 30&-8k - 9&-24k - 29&20k + 23\\
-8k - 8&-24k - 31&22k + 25&-14k - 18&24k + 32\\
-8k - 13&-24k - 33&-8k - 7&16k + 19&24k + 34\\
-24k - 35&24k + 36&-10k - 14&-4k - 4&14k + 17\\
24k + 26&24k + 38&-16k - 21&-8k - 6&-24k - 37\\
24k + 40&-24k - 39&-4k - 3&22k + 24&-18k - 22\\
\end{smallmatrix}\right] \]

and for each $0 \leq r \leq k-1$ let
\[A_r^T
\left[\begin{smallmatrix}
8k - 4r + 5&8k + 2r + 14&-16k + 2r - 18&24k + 2r + 41&-24k - 2r - 42\\
-4k + 4r - 2&-10k - 2r - 15&14k - 2r + 16&-26k - 2r - 41&26k + 2r + 42\\
8k - 4r + 4&16k + 2r + 22&-24k + 2r - 25&28k + 2r + 41&-28k - 2r - 42\\
-4k + 4r - 1&-18k - 2r - 23&22k - 2r + 23&-30k - 2r - 41&30k + 2r + 42\\
-8k + 4r - 3&-8k - 2r - 15&16k - 2r + 17&-32k - 2r - 41&32k + 2r + 42\\
4k - 4r&10k + 2r + 16&-14k + 2r - 15&34k + 2r + 41&-34k - 2r - 42\\
-8k + 4r - 2&-16k - 2r - 23&24k - 2r + 24&-36k - 2r - 41&36k + 2r + 42\\
4k - 4r - 1&18k + 2r + 24&-22k + 2r - 22&38k + 2r + 41&-38k - 2r - 42\end{smallmatrix}\right].
\]

The row and column sums of $A$ and all of the $A_r$ are all zero, hence the row and column sums of $H$ are also all zero.

The entries in $A$ cover the
intervals $I_{0 } = [1,2]$, $I_{2 } = [4k + 3,4k + 5]$, $I_{4 } = [8k + 6,8k + 13]$, $I_{18 } = [22k + 24,22k + 25]$, $I_{20 } = [24k + 26,24k + 40]$, $I_{16 } = \{20k + 23\}$, $I_{12 } = [16k + 19,16k + 21]$, $I_{14 } = \{18k + 22\}$, $I_{8 } = [12k + 15,12k + 16]$, $I_{6 } = \{10k + 14\}$ and $I_{10 } = [14k + 17,14k + 18]$.

For $0 \leq r \leq n-1$, The entries in $A_r$ cover the
intervals $[8k - 4r + 2,8k - 4r + 5]$, $[4k - 4r - 1,4k - 4r + 2]$, $[8k + 2r + 14,8k + 2r + 15]$, $[24k - 2r + 24,24k - 2r + 25]$, $[38k + 2r + 41,38k + 2r + 42]$, $[24k + 2r + 41,24k + 2r + 42]$, $[26k + 2r + 41,26k + 2r + 42]$, $[34k + 2r + 41,34k + 2r + 42]$, $[28k + 2r + 41,28k + 2r + 42]$, $[22k - 2r + 22,22k - 2r + 23]$, $[10k + 2r + 15,10k + 2r + 16]$, $[14k - 2r + 15,14k - 2r + 16]$, $[36k + 2r + 41,36k + 2r + 42]$, $[18k + 2r + 23,18k + 2r + 24]$, $[30k + 2r + 41,30k + 2r + 42]$, $[16k - 2r + 17,16k - 2r + 18]$, $[16k + 2r + 22,16k + 2r + 23]$ and $[32k + 2r + 41,32k + 2r + 42]$.
Considering $0 \leq r \leq k - 1$, these tiles cover the 
intervals $I_{3 } = [4k + 6,8k + 5]$, $I_{1 } = [3,4k + 2]$, $I_{5 } = [8k + 14,10k + 13]$, $I_{19 } = [22k + 26,24k + 25]$, $I_{28 } = [38k + 41,40k + 40]$, $I_{21 } = [24k + 41,26k + 40]$, $I_{22 } = [26k + 41,28k + 40]$, $I_{26 } = [34k + 41,36k + 40]$, $I_{23 } = [28k + 41,30k + 40]$, $I_{17 } = [20k + 24,22k + 23]$, $I_{7 } = [10k + 15,12k + 14]$, $I_{9 } = [12k + 17,14k + 16]$, $I_{27 } = [36k + 41,38k + 40]$, $I_{15 } = [18k + 23,20k + 22]$, $I_{24 } = [30k + 41,32k + 40]$, $I_{11 } = [14k + 19,16k + 18]$, $I_{13 } = [16k + 22,18k + 21]$ and $I_{25 } = [32k + 41,34k + 40]$.

Concatenating these intervals, we have that $sup(H)= [1,40k + 40] = I_0 I_1 \cdots I_{28}$.
Therefore, H is a $(5 \times 8k+8)$ integer Heffter array.

\bigskip\noindent
{\bf n $\equiv$ 1 (mod 8)}. We write $n=8k+9$. Let

\[ A^T = \left[\begin{smallmatrix}
1&-2&3&40k + 45&40k + 44\\
-8k - 11&-12k - 14&-24k - 28&24k + 29&20k + 24\\
-8k - 10&-10k - 13&18k + 22&-24k - 30&24k + 31\\
24k + 33&-18k - 20&-24k - 32&-4k - 7&22k + 26\\
-24k - 34&24k + 27&-8k - 9&-16k - 19&24k + 35\\
-4k - 5&24k + 37&-18k - 21&-24k - 36&22k + 25\\
20k + 23&24k + 39&-12k - 16&-8k - 8&-24k - 38\\
-24k - 40&-8k - 12&-8k - 6&24k + 41&16k + 17\\
24k + 43&-24k - 42&-4k - 4&-12k - 15&16k + 18\end{smallmatrix}\right]. \]

and for each $0 \leq r \leq k-1$ let
\[ A_r^T = \left[\begin{smallmatrix}
8k - 4r + 7&16k + 2r + 20&-24k + 2r - 26&24k + 2r + 44&-24k - 2r - 45\\
-4k + 4r - 2&-18k - 2r - 23&22k - 2r + 24&-26k - 2r - 44&26k + 2r + 45\\
4k - 4r + 3&10k + 2r + 14&-14k + 2r - 16&28k + 2r + 44&-28k - 2r - 45\\
-8k + 4r - 4&-8k - 2r - 13&16k - 2r + 16&-30k - 2r - 44&30k + 2r + 45\\
-8k + 4r - 5&-16k - 2r - 21&24k - 2r + 25&-32k - 2r - 44&32k + 2r + 45\\
4k - 4r&18k + 2r + 24&-22k + 2r - 23&34k + 2r + 44&-34k - 2r - 45\\
-4k + 4r - 1&-10k - 2r - 15&14k - 2r + 15&-36k - 2r - 44&36k + 2r + 45\\
8k - 4r + 2&8k + 2r + 14&-16k + 2r - 15&38k + 2r + 44&-38k - 2r - 45\end{smallmatrix}\right].\]

Then A has column sums
 $[80k + 91,\, 0,\, 0,\, 0,\, 0,\, 0,\, 0,\, 0,\, 0]$ and row sums $[0,\, 0,\, -80k - 91,\, 0,\, 160k + 182]$ while
 each $A_r$ has all row and column sums zero.  Thus the row and column sums of $H$ are all congruent to $0 \pmod {10n+1}$ as required.

The entries in $A$ cover the
intervals $I_{0 } = [1,3]$, $I_{4 } = \{4k + 7\}$, $I_{2 } = [4k + 4,4k + 5]$, $I_{6 } = \{8k + 6\}$, $I_{8 } = [8k + 8,8k + 12]$, $I_{19 } = [20k + 23,20k + 24]$, $I_{15 } = [16k + 17,16k + 19]$, $I_{17 } = [18k + 20,18k + 22]$, $I_{12 } = [12k + 14,12k + 16]$, $I_{10 } = \{10k + 13\}$, $I_{23 } = [24k + 27,24k + 43]$, $I_{32 } = [40k + 44,40k + 45]$ and $I_{21 } = [22k + 25,22k + 26]$.
a  tile with all row and column sums zero. 

The entries in $A_r$ cover the
intervals $[4k - 4r,4k - 4r + 3]$, $\{8k - 4r + 2\}$, $\{8k - 4r + 7\}$, $[8k - 4r + 4,8k - 4r + 5]$, $[8k + 2r + 13,8k + 2r + 14]$, $[26k + 2r + 44,26k + 2r + 45]$, $[14k - 2r + 15,14k - 2r + 16]$, $[36k + 2r + 44,36k + 2r + 45]$, $[16k + 2r + 20,16k + 2r + 21]$, $[38k + 2r + 44,38k + 2r + 45]$, $[24k - 2r + 25,24k - 2r + 26]$, $[30k + 2r + 44,30k + 2r + 45]$, $[16k - 2r + 15,16k - 2r + 16]$, $[24k + 2r + 44,24k + 2r + 45]$, $[34k + 2r + 44,34k + 2r + 45]$, $[18k + 2r + 23,18k + 2r + 24]$, $[22k - 2r + 23,22k - 2r + 24]$, $[28k + 2r + 44,28k + 2r + 45]$, $[32k + 2r + 44,32k + 2r + 45]$ and $[10k + 2r + 14,10k + 2r + 15]$.
Considering $0 \leq r \leq k - 1$, these tiles cover the 
intervals $I_{1 } = [4,4k + 3]$, $I_{30 } = [8k + 13,10k + 12]$, $I_{16 } = [26k + 44,28k + 43]$, $I_{31 } = [12k + 17,14k + 16]$, $I_{22 } = [36k + 44,38k + 43]$, $I_{27 } = [16k + 20,18k + 19]$, $I_{14 } = [38k + 44,40k + 43]$, $I_{24 } = [22k + 27,24k + 26]$, $I_{29 } = [30k + 44,32k + 43]$, $I_{18 } = [14k + 17,16k + 16]$, $I_{20 } = [24k + 44,26k + 43]$, $I_{26 } = [34k + 44,36k + 43]$, $I_{28 } = [18k + 23,20k + 22]$, $I_{11 } = [20k + 25,22k + 24]$, $I_{3 } = [28k + 44,30k + 43]$, $I_{5 } = [32k + 44,34k + 43]$ and $I_{7 } = [10k + 14,12k + 13]$.
Additionally, we split the intervals $\{8k - 4r + 2\}$, $\{8k - 4r + 7\}$ and $[8k - 4r + 4,8k - 4r + 5]$ into the sequences 
$(4k + 6, 4k + 10, \cdots, 8k + 2), 
 (4k + 8, 4k + 12, \cdots, 8k + 4), 
 (4k + 9, 4k + 13, \cdots, 8k + 5), $
and
$(4k + 11, 4k + 15, \cdots, 8k + 7), $
 and rejoin them into the intervals $I_{3 } = \{4k + 6\}$, $I_{5 } = [4k + 8,8k + 5]$ and $I_{7 } = \{8k + 7\}$.

Concatenating these intervals, we have covered $[1,40k + 45] = I_0 I_1 \cdots I_{32}$. 
 This construction hence yields a Heffter array
$H(5, 8k+9)$ for all $k \geq 0$.

\bigskip\noindent
{\bf n $\equiv$ 2 (mod 8)}. We write $n=8k+10$. Let
\[ A^T = \left[\begin{smallmatrix}
1&-2&3&40k + 49&40k + 50\\
-8k - 12&-12k - 15&24k + 32&-24k - 31&20k + 26\\
-8k - 11&-14k - 19&22k + 29&-24k - 33&24k + 34\\
-12k - 16&20k + 25&-24k - 35&-8k - 10&24k + 36\\
24k + 38&20k + 27&-4k - 7&-24k - 37&-16k - 21\\
24k + 40&-14k - 18&-4k - 6&18k + 23&-24k - 39\\
16k + 20&24k + 42&-24k - 41&-8k - 13&-8k - 8\\
-8k - 9&-24k - 43&24k + 44&-16k - 22&24k + 30\\
-4k - 4&-24k - 45&-10k - 14&24k + 46&14k + 17\\
-24k - 47&24k + 48&-4k - 5&22k + 28&-18k - 24\end{smallmatrix}\right], \]
and 
for each $0 \leq r \leq k-1$ let
\[ A_r^T = \left[\begin{smallmatrix}
8k - 4r + 6&8k + 2r + 14&-16k + 2r - 19&24k + 2r + 49&-24k - 2r - 50\\
-8k + 4r - 7&-16k - 2r - 23&24k - 2r + 29&-26k - 2r - 49&26k + 2r + 50\\
4k - 4r + 2&10k + 2r + 15&-14k + 2r - 16&28k + 2r + 49&-28k - 2r - 50\\
-4k + 4r - 3&-18k - 2r - 25&22k - 2r + 27&-30k - 2r - 49&30k + 2r + 50\\
-8k + 4r - 4&-8k - 2r - 15&16k - 2r + 18&-32k - 2r - 49&32k + 2r + 50\\
8k - 4r + 5&16k + 2r + 24&-24k + 2r - 28&34k + 2r + 49&-34k - 2r - 50\\
-4k + 4r&-10k - 2r - 16&14k - 2r + 15&-36k - 2r - 49&36k + 2r + 50\\
4k - 4r + 1&18k + 2r + 26&-22k + 2r - 26&38k + 2r + 49&-38k - 2r - 50\end{smallmatrix}\right]. \]

Then A has column sums
$[80k + 101,\, 0,\, 0,\, 0,\, 0,\, 0,\, 0,\, 0,\, 0,\, 0]$ and row sums $[0,\, 0,\, 0,\, 0,\, 80k + 101]$ while 
 each $A_r$ has all row and column sums zero.  Thus the row and column sums of $H$ are all congruent to $0 \pmod {10n+1}$ as required.

The entries in $A$ cover the
intervals $I_{0 } = [1,3]$, $I_{2 } = [4k + 4,4k + 7]$, $I_{4 } = [8k + 8,8k + 13]$, $I_{10 } = [14k + 17,14k + 19]$, $I_{29 } = [40k + 49,40k + 50]$, $I_{14 } = [18k + 23,18k + 24]$, $I_{20 } = [24k + 30,24k + 48]$, $I_{16 } = [20k + 25,20k + 27]$, $I_{8 } = [12k + 15,12k + 16]$, $I_{12 } = [16k + 20,16k + 22]$, $I_{18 } = [22k + 28,22k + 29]$ and $I_{6 } = \{10k + 14\}$.

The entries in $A_r$ cover the
intervals $[4k - 4r,4k - 4r + 3]$, $[8k - 4r + 4,8k - 4r + 7]$, $[8k + 2r + 14,8k + 2r + 15]$, $[24k - 2r + 28,24k - 2r + 29]$, $[28k + 2r + 49,28k + 2r + 50]$, $[14k - 2r + 15,14k - 2r + 16]$, $[32k + 2r + 49,32k + 2r + 50]$, $[16k - 2r + 18,16k - 2r + 19]$, $[30k + 2r + 49,30k + 2r + 50]$, $[38k + 2r + 49,38k + 2r + 50]$, $[24k + 2r + 49,24k + 2r + 50]$, $[36k + 2r + 49,36k + 2r + 50]$, $[10k + 2r + 15,10k + 2r + 16]$, $[16k + 2r + 23,16k + 2r + 24]$, $[34k + 2r + 49,34k + 2r + 50]$, $[22k - 2r + 26,22k - 2r + 27]$, $[26k + 2r + 49,26k + 2r + 50]$ and $[18k + 2r + 25,18k + 2r + 26]$.
Considering $0 \leq r \leq k - 1$, these tiles cover the 
intervals $I_{1 } = [4,4k + 3]$, $I_{3 } = [4k + 8,8k + 7]$, $I_{5 } = [8k + 14,10k + 13]$, $I_{19 } = [22k + 30,24k + 29]$, $I_{23 } = [28k + 49,30k + 48]$, $I_{9 } = [12k + 17,14k + 16]$, $I_{25 } = [32k + 49,34k + 48]$, $I_{11 } = [14k + 20,16k + 19]$, $I_{24 } = [30k + 49,32k + 48]$, $I_{28 } = [38k + 49,40k + 48]$, $I_{21 } = [24k + 49,26k + 48]$, $I_{27 } = [36k + 49,38k + 48]$, $I_{7 } = [10k + 15,12k + 14]$, $I_{13 } = [16k + 23,18k + 22]$, $I_{26 } = [34k + 49,36k + 48]$, $I_{17 } = [20k + 28,22k + 27]$, $I_{22 } = [26k + 49,28k + 48]$ and $I_{15 } = [18k + 25,20k + 24]$.

Concatenating these intervals, we have covered $[1,40k + 50] = I_0 I_1 \cdots I_{29}$.
This construction hence yields a Heffter array
$H(5, 8k+10)$ for all $k \geq 0$.

\bigskip\noindent
{\bf n $\equiv$ 3 (mod 8)}. We write $n=8k+11$. Let
\[ A^T = \left[\begin{smallmatrix}
8k + 13&10k + 16&-18k - 28&-2&1\\
-4k - 7&-12k - 19&16k + 25&-24k - 36&24k + 37\\
-8k - 12&-12k - 20&-24k - 38&20k + 31&24k + 39\\
-4k - 5&22k + 33&-24k - 40&24k + 41&-18k - 29\\
-16k - 26&-8k - 10&24k + 35&24k + 43&-24k - 42\\
-4k - 6&24k + 45&-10k - 17&-24k - 44&14k + 22\\
16k + 24&-24k - 46&24k + 47&-8k - 14&-8k - 11\\
22k + 32&-24k - 48&-18k - 30&-4k - 3&24k + 49\\
-24k - 50&-16k - 27&24k + 34&24k + 51&-8k - 8\\
-10k - 18&24k + 53&14k + 21&-24k - 52&-4k - 4\\
24k + 55&16k + 23&-8k - 9&-8k - 15&-24k - 54\end{smallmatrix}\right], \]
and 
for each $0 \leq r \leq k-1$, let 
\[ A_r = \left[\begin{smallmatrix}
4k - 4r + 1&18k + 2r + 31&-22k + 2r - 31&24k + 2r + 56&-24k - 2r - 57\\
-8k + 4r - 6&-16k - 2r - 28&24k - 2r + 33&-26k - 2r - 56&26k + 2r + 57\\
4k - 4r + 2&10k + 2r + 19&-14k + 2r - 20&28k + 2r + 56&-28k - 2r - 57\\
-8k + 4r - 7&-8k - 2r - 16&16k - 2r + 22&-30k - 2r - 56&30k + 2r + 57\\
-4k + 4r + 1&-18k - 2r - 32&22k - 2r + 30&-32k - 2r - 56&32k + 2r + 57\\
8k - 4r + 4&16k + 2r + 29&-24k + 2r - 32&34k + 2r + 56&-34k - 2r - 57\\
-4k + 4r&-10k - 2r - 20&14k - 2r + 19&-36k - 2r - 56&36k + 2r + 57\\
8k - 4r + 5&8k + 2r + 17&-16k + 2r - 21&38k + 2r + 56&-38k - 2r - 57\end{smallmatrix}\right]. \]

In this case all the row sums and all the column sums of $A$ and every $A_r$ are equal to zero, hence clearly $H$ has row and column sums equal to zero.

The entries in $A$ cover the
intervals $I_{0 } = [1,2]$, $I_{2 } = [4k + 3,4k + 7]$, $I_{4 } = [8k + 8,8k + 15]$, $I_{18 } = [22k + 32,22k + 33]$, $I_{6 } = [10k + 16,10k + 18]$, $I_{8 } = [12k + 19,12k + 20]$, $I_{14 } = [18k + 28,18k + 30]$, $I_{10 } = [14k + 21,14k + 22]$, $I_{16 } = \{20k + 31\}$, $I_{12 } = [16k + 23,16k + 27]$ and $I_{20 } = [24k + 34,24k + 55]$.

 The entries in $A_r$ cover the
intervals $[8k - 4r + 4,8k - 4r + 7]$, $[4k - 4r - 1,4k - 4r + 2]$, $[8k + 2r + 16,8k + 2r + 17]$, $[28k + 2r + 56,28k + 2r + 57]$, $[22k - 2r + 30,22k - 2r + 31]$, $[30k + 2r + 56,30k + 2r + 57]$, $[26k + 2r + 56,26k + 2r + 57]$, $[18k + 2r + 31,18k + 2r + 32]$, $[38k + 2r + 56,38k + 2r + 57]$, $[24k + 2r + 56,24k + 2r + 57]$, $[34k + 2r + 56,34k + 2r + 57]$, $[14k - 2r + 19,14k - 2r + 20]$, $[24k - 2r + 32,24k - 2r + 33]$, $[10k + 2r + 19,10k + 2r + 20]$, $[32k + 2r + 56,32k + 2r + 57]$, $[16k - 2r + 21,16k - 2r + 22]$, $[16k + 2r + 28,16k + 2r + 29]$ and $[36k + 2r + 56,36k + 2r + 57]$.
Considering $0 \leq r \leq k - 1$, these tiles cover the 
intervals $I_{3 } = [4k + 8,8k + 7]$, $I_{1 } = [3,4k + 2]$, $I_{5 } = [8k + 16,10k + 15]$, $I_{23 } = [28k + 56,30k + 55]$, $I_{17 } = [20k + 32,22k + 31]$, $I_{24 } = [30k + 56,32k + 55]$, $I_{22 } = [26k + 56,28k + 55]$, $I_{15 } = [18k + 31,20k + 30]$, $I_{28 } = [38k + 56,40k + 55]$, $I_{21 } = [24k + 56,26k + 55]$, $I_{26 } = [34k + 56,36k + 55]$, $I_{9 } = [12k + 21,14k + 20]$, $I_{19 } = [22k + 34,24k + 33]$, $I_{7 } = [10k + 19,12k + 18]$, $I_{25 } = [32k + 56,34k + 55]$, $I_{11 } = [14k + 23,16k + 22]$, $I_{13 } = [16k + 28,18k + 27]$ and $I_{27 } = [36k + 56,38k + 55]$.

Concatenating these intervals, we have covered $[1,40k + 55] = I_0 I_1 \cdots I_{28}$.
Hence, this construction yields an integer Heffter array $H(5, 8k+11)$ for all $k \geq 0$.

\bigskip\noindent
{\bf n $\equiv$ 4 (mod 8)}. We write $n=8k+12$. Let
\[ A^T = \left[\begin{smallmatrix}
8k + 14&8k + 15&-16k - 28&1&-2\\
-4k - 7&-8k - 16&12k + 22&-24k - 39&24k + 40\\
-8k - 13&-12k - 21&-24k - 41&24k + 42&20k + 33\\
-8k - 12&-14k - 25&22k + 36&24k + 44&-24k - 43\\
-8k - 17&16k + 27&-24k - 45&-8k - 11&24k + 46\\
14k + 24&-10k - 19&24k + 48&-24k - 47&-4k - 6\\
-24k - 49&24k + 50&24k + 38&-16k - 29&-8k - 10\\
-24k - 51&-4k - 5&24k + 52&22k + 35&-18k - 31\\
-8k - 9&24k + 54&-24k - 53&-8k - 18&16k + 26\\
14k + 23&24k + 56&-24k - 55&-4k - 4&-10k - 20\\
24k + 37&-24k - 57&24k + 58&-8k - 8&-16k - 30\\
24k + 60&-24k - 59&-18k - 32&22k + 34&-4k - 3\end{smallmatrix}\right], \]
and 
for each $0 \leq r \leq k-1$, let 
\[ A_r^T = \left[\begin{smallmatrix}
8k - 4r + 7&8k + 2r + 19&-16k + 2r - 25&24k + 2r + 61&-24k - 2r - 62\\
-4k + 4r - 2&-10k - 2r - 21&14k - 2r + 22&-26k - 2r - 61&26k + 2r + 62\\
8k - 4r + 6&16k + 2r + 31&-24k + 2r - 36&28k + 2r + 61&-28k - 2r - 62\\
-4k + 4r - 1&-18k - 2r - 33&22k - 2r + 33&-30k - 2r - 61&30k + 2r + 62\\
-8k + 4r - 5&-8k - 2r - 20&16k - 2r + 24&-32k - 2r - 61&32k + 2r + 62\\
4k - 4r&10k + 2r + 22&-14k + 2r - 21&34k + 2r + 61&-34k - 2r - 62\\
-8k + 4r - 4&-16k - 2r - 32&24k - 2r + 35&-36k - 2r - 61&36k + 2r + 62\\
4k - 4r - 1&18k + 2r + 34&-22k + 2r - 32&38k + 2r + 61&-38k - 2r - 62\end{smallmatrix}\right], \]

In this case all the row sums and all the column sums of $A$ and every $A_r$ are equal to zero, hence clearly $H$ has row and column sums equal to zero.

The entries in $A$ cover the
intervals $I_{0 } = [1,2]$, $I_{2 } = [4k + 3,4k + 7]$, $I_{4 } = [8k + 8,8k + 18]$, $I_{10 } = [14k + 23,14k + 25]$, $I_{8 } = [12k + 21,12k + 22]$, $I_{16 } = \{20k + 33\}$, $I_{14 } = [18k + 31,18k + 32]$, $I_{12 } = [16k + 26,16k + 30]$, $I_{6 } = [10k + 19,10k + 20]$, $I_{18 } = [22k + 34,22k + 36]$ and $I_{20 } = [24k + 37,24k + 60]$.

The entries in $A_r$ cover the
intervals $[8k - 4r + 4,8k - 4r + 7]$, $[4k - 4r - 1,4k - 4r + 2]$, $[8k + 2r + 19,8k + 2r + 20]$, $[22k - 2r + 32,22k - 2r + 33]$, $[34k + 2r + 61,34k + 2r + 62]$, $[30k + 2r + 61,30k + 2r + 62]$, $[38k + 2r + 61,38k + 2r + 62]$, $[24k - 2r + 35,24k - 2r + 36]$, $[24k + 2r + 61,24k + 2r + 62]$, $[26k + 2r + 61,26k + 2r + 62]$, $[10k + 2r + 21,10k + 2r + 22]$, $[32k + 2r + 61,32k + 2r + 62]$, $[16k + 2r + 31,16k + 2r + 32]$, $[16k - 2r + 24,16k - 2r + 25]$, $[36k + 2r + 61,36k + 2r + 62]$, $[28k + 2r + 61,28k + 2r + 62]$, $[14k - 2r + 21,14k - 2r + 22]$ and $[18k + 2r + 33,18k + 2r + 34]$.
Considering $0 \leq r \leq k - 1$, these tiles cover the 
intervals $I_{3 } = [4k + 8,8k + 7]$, $I_{1 } = [3,4k + 2]$, $I_{5 } = [8k + 19,10k + 18]$, $I_{17 } = [20k + 34,22k + 33]$, $I_{26 } = [34k + 61,36k + 60]$, $I_{24 } = [30k + 61,32k + 60]$, $I_{28 } = [38k + 61,40k + 60]$, $I_{19 } = [22k + 37,24k + 36]$, $I_{21 } = [24k + 61,26k + 60]$, $I_{22 } = [26k + 61,28k + 60]$, $I_{7 } = [10k + 21,12k + 20]$, $I_{25 } = [32k + 61,34k + 60]$, $I_{13 } = [16k + 31,18k + 30]$, $I_{11 } = [14k + 26,16k + 25]$, $I_{27 } = [36k + 61,38k + 60]$, $I_{23 } = [28k + 61,30k + 60]$, $I_{9 } = [12k + 23,14k + 22]$ and $I_{15 } = [18k + 33,20k + 32]$.

Concatenating these intervals, we have covered $[1,40k + 60] = I_0 I_1 \cdots I_{28}$.
 This construction hence yields an integer Heffter array
$H(5, 8k+12)$ for all $k \geq 0$.

\bigskip\noindent
{\bf n $\equiv$ 5 (mod 8)}. We write $n=8k+13$. Let
\[ A^T = \left[\begin{smallmatrix}
1&-2&3&40k + 65&40k + 64\\
-8k - 15&-12k - 20&20k + 34&-24k - 40&24k + 41\\
-8k - 14&-10k - 18&18k + 31&-24k - 42&24k + 43\\
-4k - 9&-18k - 29&-24k - 44&24k + 45&22k + 37\\
-8k - 13&-16k - 27&24k + 47&24k + 39&-24k - 46\\
-4k - 7&-18k - 30&24k + 49&-24k - 48&22k + 36\\
20k + 33&-24k - 50&-8k - 12&-12k - 22&24k + 51\\
16k + 25&-24k - 52&-8k - 16&-8k - 10&24k + 53\\
24k + 55&-4k - 6&-12k - 21&-24k - 54&16k + 26\\
24k + 57&-24k - 56&-16k - 28&-8k - 11&24k + 38\\
-18k - 32&22k + 35&-24k - 58&24k + 59&-4k - 4\\
-10k - 19&24k + 61&14k + 23&-4k - 5&-24k - 60\\
-24k - 62&24k + 63&-8k - 8&16k + 24&-8k - 17\end{smallmatrix}\right], \]
and 
for each $0 \leq r \leq k-1$, let 
\[ A_r^T = \left[\begin{smallmatrix}
8k - 4r + 9&16k + 2r + 29&-24k + 2r - 37&24k + 2r + 64&-24k - 2r - 65\\
-4k + 4r - 2&-18k - 2r - 33&22k - 2r + 34&-26k - 2r - 64&26k + 2r + 65\\
4k - 4r + 3&10k + 2r + 20&-14k + 2r - 22&28k + 2r + 64&-28k - 2r - 65\\
-8k + 4r - 6&-8k - 2r - 18&16k - 2r + 23&-30k - 2r - 64&30k + 2r + 65\\
-8k + 4r - 7&-16k - 2r - 30&24k - 2r + 36&-32k - 2r - 64&32k + 2r + 65\\
4k - 4r&18k + 2r + 34&-22k + 2r - 33&34k + 2r + 64&-34k - 2r - 65\\
-4k + 4r - 1&-10k - 2r - 21&14k - 2r + 21&-36k - 2r - 64&36k + 2r + 65\\
8k - 4r + 4&8k + 2r + 19&-16k + 2r - 22&38k + 2r + 64&-38k - 2r - 65\end{smallmatrix}\right]. \]

Then $A$ has column sums $[80k + 131,\, 0,\, 0,\, 0,\, 0,\, 0,\, 0,\, 0,\, 0,\, 0,\, 0,\, 0,\, 0]$ and row sums $[0,\, -80k - 131,\, 0,\, 0,\, 160k + 262]$ while 
 each $A_r$ has all row and column sums zero.  Thus the row and column sums of $H$ are all congruent to $0 \pmod {10n+1}$ as required.

 The entries in $A$ cover the
intervals $I_{0 } = [1,3]$, $I_{6 } = \{8k + 8\}$, $I_{2 } = [4k + 4,4k + 7]$, $I_{4 } = \{4k + 9\}$, $I_{8 } = [8k + 10,8k + 17]$, $I_{10 } = [10k + 18,10k + 19]$, $I_{14 } = \{14k + 23\}$, $I_{33 } = [40k + 64,40k + 65]$, $I_{22 } = [22k + 35,22k + 37]$, $I_{18 } = [18k + 29,18k + 32]$, $I_{16 } = [16k + 24,16k + 28]$, $I_{20 } = [20k + 33,20k + 34]$, $I_{24 } = [24k + 38,24k + 63]$ and $I_{12 } = [12k + 20,12k + 22]$.

 The entries in $A_r$ cover the
intervals $[4k - 4r,4k - 4r + 3]$, $[8k - 4r + 6,8k - 4r + 7]$, $\{8k - 4r + 4\}$, $\{8k - 4r + 9\}$, $[8k + 2r + 18,8k + 2r + 19]$, $[10k + 2r + 20,10k + 2r + 21]$, $[38k + 2r + 64,38k + 2r + 65]$, $[34k + 2r + 64,34k + 2r + 65]$, $[26k + 2r + 64,26k + 2r + 65]$, $[32k + 2r + 64,32k + 2r + 65]$, $[16k + 2r + 29,16k + 2r + 30]$, $[24k - 2r + 36,24k - 2r + 37]$, $[30k + 2r + 64,30k + 2r + 65]$, $[36k + 2r + 64,36k + 2r + 65]$, $[28k + 2r + 64,28k + 2r + 65]$, $[22k - 2r + 33,22k - 2r + 34]$, $[18k + 2r + 33,18k + 2r + 34]$, $[16k - 2r + 22,16k - 2r + 23]$, $[14k - 2r + 21,14k - 2r + 22]$ and $[24k + 2r + 64,24k + 2r + 65]$.
Considering $0 \leq r \leq k - 1$, these tiles cover the 
intervals $I_{1 } = [4,4k + 3]$, $I_{30 } = [8k + 18,10k + 17]$, $I_{26 } = [10k + 20,12k + 19]$, $I_{29 } = [38k + 64,40k + 63]$, $I_{17 } = [34k + 64,36k + 63]$, $I_{23 } = [26k + 64,28k + 63]$, $I_{28 } = [32k + 64,34k + 63]$, $I_{31 } = [16k + 29,18k + 28]$, $I_{27 } = [22k + 38,24k + 37]$, $I_{21 } = [30k + 64,32k + 63]$, $I_{19 } = [36k + 64,38k + 63]$, $I_{15 } = [28k + 64,30k + 63]$, $I_{13 } = [20k + 35,22k + 34]$, $I_{25 } = [18k + 33,20k + 32]$, $I_{3 } = [14k + 24,16k + 23]$, $I_{5 } = [12k + 23,14k + 22]$ and $I_{7 } = [24k + 64,26k + 63]$.
Additionally, we split the intervals $[8k - 4r + 6,8k - 4r + 7]$, $\{8k - 4r + 4\}$ and $\{8k - 4r + 9\}$ into the sequences 
$ (4k + 8, 4k + 12, \cdots, 8k + 4), 
 (4k + 10, 4k + 14, \cdots, 8k + 6), 
 (4k + 11, 4k + 15, \cdots, 8k + 7), $
and $ (4k + 13, 4k + 17, \cdots, 8k + 9), $
 and rejoin them into the intervals $I_{3 } = \{4k + 8\}$, $I_{5 } = [4k + 10,8k + 7]$ and $I_{7 } = \{8k + 9\}$.

Concatenating these intervals, we have covered $[1,40k + 65] = I_0 I_1 \cdots I_{33}$.
Hence, this construction yields a Heffter array $H(5, 8k+13)$ for all $k \geq 0$.

\bigskip\noindent
{\bf n $\equiv$ 6 (mod 8)}. We write $n=8k+14$. Let
\[ A^T = \left[\begin{smallmatrix}
1&-2&3&40k + 70&40k + 69\\
-8k - 16&-12k - 21&20k + 36&24k + 44&-24k - 43\\
-8k - 15&-14k - 26&22k + 40&-24k - 45&24k + 46\\
-8k - 14&-12k - 22&20k + 35&-24k - 47&24k + 48\\
-4k - 9&-16k - 29&20k + 37&-24k - 49&24k + 50\\
-4k - 8&-14k - 25&18k + 32&-24k - 51&24k + 52\\
-8k - 12&-8k - 17&16k + 28&-24k - 53&24k + 54\\
-8k - 13&-16k - 30&24k + 56&24k + 42&-24k - 55\\
-4k - 6&-24k - 57&-10k - 19&24k + 58&14k + 24\\
-4k - 7&24k + 60&-24k - 59&-18k - 33&22k + 39\\
-8k - 10&16k + 27&24k + 62&-8k - 18&-24k - 61\\
-16k - 31&-24k - 63&-8k - 11&24k + 64&24k + 41\\
24k + 66&-4k - 4&-24k - 65&14k + 23&-10k - 20\\
-24k - 67&24k + 68&-18k - 34&-4k - 5&22k + 38\end{smallmatrix}\right], \]
and 
for each $0 \leq r \leq k-1$, let 
\[ A_r^T = \left[\begin{smallmatrix}
8k - 4r + 8&8k + 2r + 19&-16k + 2r - 26&24k + 2r + 69&-24k - 2r - 70\\
-8k + 4r - 9&-16k - 2r - 32&24k - 2r + 40&-26k - 2r - 69&26k + 2r + 70\\
4k - 4r + 2&10k + 2r + 21&-14k + 2r - 22&28k + 2r + 69&-28k - 2r - 70\\
-4k + 4r - 3&-18k - 2r - 35&22k - 2r + 37&-30k - 2r - 69&30k + 2r + 70\\
-8k + 4r - 6&-8k - 2r - 20&16k - 2r + 25&-32k - 2r - 69&32k + 2r + 70\\
8k - 4r + 7&16k + 2r + 33&-24k + 2r - 39&34k + 2r + 69&-34k - 2r - 70\\
-4k + 4r&-10k - 2r - 22&14k - 2r + 21&-36k - 2r - 69&36k + 2r + 70\\
4k - 4r + 1&18k + 2r + 36&-22k + 2r - 36&38k + 2r + 69&-38k - 2r - 70\end{smallmatrix}\right]. \]

Then $A$ has column sums  $[80k + 141,\, 0,\, 0,\, 0,\, 0,\, 0,\, 0,\, 0,\, 0,\, 0,\, 0,\, 0,\, 0,\, 0]$ and row sums $[-80k - 141,\, -80k - 141,\, 80k + 141,\, 0,\, 160k + 282]$ while 
 each $A_r$ has all row and column sums zero.  Thus the row and column sums of $H$ are all congruent to $0 \pmod {10n+1}$ as required.

 The entries in $A$ cover the
intervals $I_{0 } = [1,3]$, $I_{2 } = [4k + 4,4k + 9]$, $I_{4 } = [8k + 10,8k + 18]$, $I_{8 } = [12k + 21,12k + 22]$, $I_{10 } = [14k + 23,14k + 26]$, $I_{18 } = [22k + 38,22k + 40]$, $I_{29 } = [40k + 69,40k + 70]$, $I_{12 } = [16k + 27,16k + 31]$, $I_{14 } = [18k + 32,18k + 34]$, $I_{6 } = [10k + 19,10k + 20]$, $I_{20 } = [24k + 41,24k + 68]$ and $I_{16 } = [20k + 35,20k + 37]$.

The entries in $A_r$ cover the
intervals $[4k - 4r,4k - 4r + 3]$, $[8k - 4r + 6,8k - 4r + 9]$, $[8k + 2r + 19,8k + 2r + 20]$, $[30k + 2r + 69,30k + 2r + 70]$, $[10k + 2r + 21,10k + 2r + 22]$, $[26k + 2r + 69,26k + 2r + 70]$, $[22k - 2r + 36,22k - 2r + 37]$, $[38k + 2r + 69,38k + 2r + 70]$, $[16k + 2r + 32,16k + 2r + 33]$, $[28k + 2r + 69,28k + 2r + 70]$, $[24k - 2r + 39,24k - 2r + 40]$, $[14k - 2r + 21,14k - 2r + 22]$, $[24k + 2r + 69,24k + 2r + 70]$, $[34k + 2r + 69,34k + 2r + 70]$, $[18k + 2r + 35,18k + 2r + 36]$, $[36k + 2r + 69,36k + 2r + 70]$, $[16k - 2r + 25,16k - 2r + 26]$ and $[32k + 2r + 69,32k + 2r + 70]$.
Considering $0 \leq r \leq k - 1$, these tiles cover the 
intervals $I_{1 } = [4,4k + 3]$, $I_{3 } = [4k + 10,8k + 9]$, $I_{5 } = [8k + 19,10k + 18]$, $I_{24 } = [30k + 69,32k + 68]$, $I_{7 } = [10k + 21,12k + 20]$, $I_{22 } = [26k + 69,28k + 68]$, $I_{17 } = [20k + 38,22k + 37]$, $I_{28 } = [38k + 69,40k + 68]$, $I_{13 } = [16k + 32,18k + 31]$, $I_{23 } = [28k + 69,30k + 68]$, $I_{19 } = [22k + 41,24k + 40]$, $I_{9 } = [12k + 23,14k + 22]$, $I_{21 } = [24k + 69,26k + 68]$, $I_{26 } = [34k + 69,36k + 68]$, $I_{15 } = [18k + 35,20k + 34]$, $I_{27 } = [36k + 69,38k + 68]$, $I_{11 } = [14k + 27,16k + 26]$ and $I_{25 } = [32k + 69,34k + 68]$.

Concatenating these intervals, we have covered $[1,40k + 70] = I_0 I_1 \cdots I_{29}$.
This construction hence yields a Heffter array $H(5, 8k+13)$ for all $k \geq 0$.

\bigskip\noindent
{\bf n $\equiv$ 7 (mod 8)}. We write $n=8k+7$. Let
\[ A^T = \left[\begin{smallmatrix}
8k + 9&10k + 11&-18k - 19&1&-2\\
-4k - 5&16k + 17&-12k - 13&-24k - 24&24k + 25\\
24k + 27&-24k - 26&-8k - 8&20k + 21&-12k - 14\\
-4k - 3&22k + 22&24k + 29&-18k - 20&-24k - 28\\
-24k - 31&-24k - 23&8k + 6&16k + 18&24k + 30\\
-24k - 32&24k + 33&14k + 15&-10k - 12&-4k - 4\\
24k + 35&-24k - 34&-8k - 10&16k + 16&-8k - 7\end{smallmatrix}\right], \]
and 
for each $0 \leq r \leq k-1$, let 
\[ A_r^T = \left[\begin{smallmatrix}
4k - 4r + 1&18k + 2r + 21&-22k + 2r - 21&24k + 2r + 36&-24k - 2r - 37\\
-8k + 4r - 4&-16k - 2r - 19&24k - 2r + 22&-26k - 2r - 36&26k + 2r + 37\\
4k - 4r + 2&10k + 2r + 13&-14k + 2r - 14&28k + 2r + 36&-28k - 2r - 37\\
-8k + 4r - 5&-8k - 2r - 11&16k - 2r + 15&-30k - 2r - 36&30k + 2r + 37\\
-4k + 4r + 1&-18k - 2r - 22&22k - 2r + 20&-32k - 2r - 36&32k + 2r + 37\\
8k - 4r + 2&16k + 2r + 20&-24k + 2r - 21&34k + 2r + 36&-34k - 2r - 37\\
-4k + 4r&-10k - 2r - 14&14k - 2r + 13&-36k - 2r - 36&36k + 2r + 37\\
8k - 4r + 3&8k + 2r + 12&-16k + 2r - 14&38k + 2r + 36&-38k - 2r - 37\end{smallmatrix}\right]. \]

In this case all the row sums and all the column sums of $A$ and every $A_r$ are equal to zero, hence clearly $H$ has row and column sums equal to zero.

The entries in $A$ cover the
intervals $I_{0 } = [1,2]$, $I_{2 } = [4k + 3,4k + 5]$, $I_{4 } = [8k + 6,8k + 10]$, $I_{8 } = [12k + 13,12k + 14]$, $I_{20 } = [24k + 23,24k + 35]$, $I_{6 } = [10k + 11,10k + 12]$, $I_{12 } = [16k + 16,16k + 18]$, $I_{10 } = \{14k + 15\}$, $I_{16 } = \{20k + 21\}$, $I_{14 } = [18k + 19,18k + 20]$ and $I_{18 } = \{22k + 22\}$.

The entries in $A_r$ cover the
intervals $[4k - 4r - 1,4k - 4r + 2]$, $[8k - 4r + 2,8k - 4r + 5]$, $[8k + 2r + 11,8k + 2r + 12]$, $[24k - 2r + 21,24k - 2r + 22]$, $[10k + 2r + 13,10k + 2r + 14]$, $[36k + 2r + 36,36k + 2r + 37]$, $[32k + 2r + 36,32k + 2r + 37]$, $[34k + 2r + 36,34k + 2r + 37]$, $[26k + 2r + 36,26k + 2r + 37]$, $[28k + 2r + 36,28k + 2r + 37]$, $[16k - 2r + 14,16k - 2r + 15]$, $[30k + 2r + 36,30k + 2r + 37]$, $[16k + 2r + 19,16k + 2r + 20]$, $[22k - 2r + 20,22k - 2r + 21]$, $[38k + 2r + 36,38k + 2r + 37]$, $[14k - 2r + 13,14k - 2r + 14]$, $[24k + 2r + 36,24k + 2r + 37]$ and $[18k + 2r + 21,18k + 2r + 22]$.
Considering $0 \leq r \leq k - 1$, these tiles cover the 
intervals $I_{1 } = [3,4k + 2]$, $I_{3 } = [4k + 6,8k + 5]$, $I_{5 } = [8k + 11,10k + 10]$, $I_{19 } = [22k + 23,24k + 22]$, $I_{7 } = [10k + 13,12k + 12]$, $I_{27 } = [36k + 36,38k + 35]$, $I_{25 } = [32k + 36,34k + 35]$, $I_{26 } = [34k + 36,36k + 35]$, $I_{22 } = [26k + 36,28k + 35]$, $I_{23 } = [28k + 36,30k + 35]$, $I_{11 } = [14k + 16,16k + 15]$, $I_{24 } = [30k + 36,32k + 35]$, $I_{13 } = [16k + 19,18k + 18]$, $I_{17 } = [20k + 22,22k + 21]$, $I_{28 } = [38k + 36,40k + 35]$, $I_{9 } = [12k + 15,14k + 14]$, $I_{21 } = [24k + 36,26k + 35]$ and $I_{15 } = [18k + 21,20k + 20]$.

Concatenating these intervals, we have covered $[1,40k + 35] = I_0 I_1 \cdots I_{28}$.
 This construction hence yields an integer Heffter array
$H(5, 8k+7)$ for all $k \geq 0$.  

This concludes the proof of the theorem.
\end{proof}

%%%%%%%%%%%%%%%%%%%%%%%%%%%%%%%%%%%%%%%%%%%%%%%%%%%%%%%%%%%%%%%
%%%%%%%%%%%%%%%%%%%%%%%%%%%%%%%%%%%%%%%%%%%%%%%%%%%%%%%%%%%%%%%

%%%%%%%%%%%%%%%%%%%%%%%%%%%%%%%%%%%%%%%%%%%%%%
\section {\label{extending} Constructing H($m,n$) for $m$ odd and $n$ even}
%%%%%%%%%%%%%%%%%%%%%%%%%%%%%%%%%%%%%%%%%%%%%%

As previously noted we will add additional rows to the $H(3,n)$ and to the $H(5,n)$ constructed in the previous sections to get $H(m,n)$ for odd values of $m$.  The first cases we consider  are the easiest ones; namely the cases when the  $H(3,n)$ and  $H(5,n)$  are integer and we add an even number of rows to make an integer $H(m,n)$.  It is easy to see that this implies that $n \equiv 0 \pmod{4}$.
The next theorem considers this case.

\begin{theorem}\label{oddby4k_theorem} There exists an integer Heffter array $H(m,n)$ for all odd $m\geq 3$ and all $n \equiv 0 \pmod{4}$.
\end{theorem}

\begin{proof}First assume that  $m \equiv 1 \pmod 4$.  Write $m =4k+5$ with $k \geq 0$.  By Theorem \ref{5byn_theorem} there exists an integer Heffter array $A = H(5,n)$.  Additionally, let $B$ be an integer shiftable  Heffter array $H(4k,n)$ which exists by Theorem \ref{evenbyeventhm}.
Then, it is easy to see that the array 

$$H= \begin{array}{|c|} \hline
A    \\ \hline
\hspace*{.5in}B \pm 5n \hspace*{.5in} \\ \hline
\end{array}
$$
 is an integer $H(4k+5,n)$.

\medskip
Now assume that  $m \equiv 3 \pmod 4$.  Write $m =4k+3$ with $k \geq 0$.  By Theorem \ref{3byn_theorem} there exists an integer Heffter array $A = H(3,n)$.  Additionally, let $B$ be an integer shiftable  Heffter array $H(4k,n)$ which exists by Theorem \ref{evenbyeventhm}.
Then, it is again clear that the array 

$$H= \begin{array}{|c|} \hline
A    \\ \hline
\hspace*{.5in}B \pm 3n \hspace*{.5in} \\ \hline
\end{array}
$$
 is an integer $H(4k+3,n)$.
\end{proof}

Next we wish to consider the case when $n \equiv 2 \pmod 4$. So the $H(3,n)$ and  $H(5,n)$ constructed earlier are necessarily noninteger and if we add an even number of additional rows the resulting array will be noninteger also.    This is a harder case for the following reason.  If we just try to mimic the previous theorem and place $A$, an $H(3,n)$ from Theorem  \ref{3byn_theorem} above a shiftable Heffter array $H(4k,n)$ we will have that at least one of the rows (and one of the columns) of $A$ add to 0 modulo $6n+1$.   Hence when these rows are placed in the bigger array (which is $(4k+3) \times n$)  they will probably not add to 0 modulo $2(4k+3)n+1$ and hence it won't be a Heffter array.  In other words, since the modulus changes when we construct a different size array, if a row or column adds to zero in the old modulus, this in no ways implies it adds to zero in the new one.  We address this problem in the next theorem and solve it via a so called {\em variable tile}.  (Again, the interested reader is referred to \cite{Tom-thesis} for a discussion of how the tiles in this section were originally constructed.)

\begin{theorem}\label{oddby4k+2_theorem} There exists a Heffter array $H(m,n)$ for all odd $m\geq 3$ and all $n \equiv 2 \pmod{4}$.
\end{theorem}

\begin{proof}
First assume that $m \equiv 1 \pmod 4$ with $m\geq 5$.  We note that if $m=5$, this case is solved in Theorem \ref{5byn_theorem}. So assume that  $m \equiv 1 \pmod 4, m\geq 9$ and $n \equiv 2 \pmod 4, n \geq 6$.  Write $m = 4s+1$ with $s \geq 2$ and $n =4k+2$ with $k \geq 1$.  We begin this construction with a $9 \times 6$ variable tile

\[ A = \left[\begin{smallmatrix}
x - 10&-x + 9&5&x - 3&x - 4&4\\
-x + 6&x - 8&-x + 16&x - 19&-x + 20&x - 15\\
6&-x + 23&-x + 7&-1&x - 21&x - 14\\
x&x - 22&x - 11&-x + 18&-x + 2&-x + 13\\
x - 1&-2&x - 17&-x + 5&3&-x + 12\\
x - 47&-x + 46&x - 45&-x + 44&x - 37&-x + 39\\
-x + 41&x - 40&-x + 36&x - 38&-x + 43&x - 42\\
-x + 35&x - 34&-x + 33&x - 32&-x + 25&x - 27\\
x - 29&-x + 28&x - 24&-x + 26&x - 31&-x + 30\end{smallmatrix}\right].
\]
Note that  $A$ has row sums $[2x + 1,\, 0,\, 0,\, 0,\, 0,\, 0,\, 0,\, 0,\, 0]$ and column sums $[2x + 1,\, 0,\, 0,\, 0,\, 0,\, 0]$.   So by setting the variable $x$ to different values (as a function of $m$ and $n$) we can adjust these row and column sums to be zero modulo $2mn+1$.   Let $x = mn=(4s+1)(4k+2)$.  Now we have that the entries in $A$ cover the
intervals $I_{0 } = [1,6]$ (the lowest possible numbers)  and $I_{8 } = [16s k + 8s + 4k - 45,16s k + 8s + 4k + 2]$ (the top numbers).

For each $0 \leq r \leq k-2$ let
\[
 A_r = \left[\begin{smallmatrix}
4r + 7&-4r - 8&-4r - 9&4r + 10\\
6k - 2r&-10k + 2r + 4&-6k + 2r + 1&10k - 2r - 5\\
-6k - 2r - 1&10k + 2r - 3&6k + 2r + 2&-10k - 2r + 2\\
12k + 24r - 5&-12k - 24r + 3&-12k - 24r + 1&12k + 24r + 1\\
-12k - 24r + 4&12k + 24r - 2&12k + 24r&-12k - 24r - 2\\
12k + 24r + 3&-12k - 24r - 4&-12k - 24r - 7&12k + 24r + 8\\
-12k - 24r - 5&12k + 24r + 6&12k + 24r + 9&-12k - 24r - 10\\
12k + 24r + 11&-12k - 24r - 12&12k + 24r + 18&-12k - 24r - 17\\
-12k - 24r - 14&12k + 24r + 16&-12k - 24r - 15&12k + 24r + 13\end{smallmatrix}\right], \]

We note here that a discussion of how the tiles $A$ and $A_r$ were found appears in \cite{Tom-thesis}.
Each $A_r$ has all row and column sums equal to zero.  
The entries in $A_r$ cover the
intervals $[4r + 7,4r + 10]$, $[6k - 2r - 1,6k - 2r]$, $[6k + 2r + 1,6k + 2r + 2]$, $[10k + 2r - 3,10k + 2r - 2]$, $[10k - 2r - 5,10k - 2r - 4]$ and $[12k + 24r - 5,12k + 24r + 18]$.
Considering $0 \leq r \leq k - 2$, these tiles cover the 
intervals $I_{1 } = [7,4k + 2]$, $I_{2 } = [4k + 3,6k]$, $I_{3 } = [6k + 1,8k - 2]$, $I_{5 } = [10k - 3,12k - 6]$, $I_{4 } = [8k - 1,10k - 4]$ and $I_{6 } = [12k - 5,36k - 30]$.

Finally let  $E_0$ be a $(4s - 8 \times 4k + 2)$ shiftable integer Heffter array
and $E = E_0 \pm (36k - 30)$.  Then $sup(E)=I_{7 } = [36k - 29,8{\left(2k + 1\right)} {\left(s - 2\right)} + 36k - 30]$. Note that $sup(E)$ contains the values just under the largest values in $sup(A)$.
Concatenating these intervals, we have covered $[1,16s k + 8s + 4k + 2] =[1,mn]  = I_0 I_1 \cdots I_{8}$.

Now, put these all together as follows to construct an array $H$ were
\[H= 
  \begin{array}{|c|c|c|c|} \hline
   A & A_0 &\hspace*{.3in} \cdots\hspace*{.3in} & A_{k-2} \\
   \hline
   \multicolumn{4}{|c|}{E} \\ \hline
  \end{array}
\]
 It is clear that each row and column sum is 0 modulo $mn +1$ and thus $H$ is an  $H(4s+1,4k+2)$

\medskip
We now turn to the case when 
$m \equiv 3 \pmod 4$ with $m\geq 3$. This case will proceed in a similar manner to the previous case.  We note that if $m=3$, this case is solved in Theorem \ref{3byn_theorem}. So assume that  $m \equiv 3 \pmod 4, m\geq 7$ and $n \equiv 2 \pmod 4, n \geq 6$.   Write $m = 4s+3$ with $s \geq 2$ and $n =4k+2$ with $k \geq 1$.  We begin this construction with a $7 \times 6$ variable tile

\[ A = \left[\begin{smallmatrix}
2&x - 2&x - 4&3&x - 5&-x + 7\\
x - 1&-x + 8&-x + 3&x - 9&-5&4\\
x&-6&1&-x + 6&-x + 10&x - 11\\
x - 35&-x + 34&x - 33&-x + 32&x - 25&-x + 27\\
-x + 29&x - 28&-x + 24&x - 26&-x + 31&x - 30\\
-x + 23&x - 22&-x + 21&x - 20&-x + 13&x - 15\\
x - 17&-x + 16&x - 12&-x + 14&x - 19&-x + 18\end{smallmatrix}\right]. \]

Note that $A$ has  row sums $[2x + 1,\, 0,\, 0,\, 0,\, 0,\, 0,\, 0]$ and column sums $[2x + 1,\, 0,\, 0,\, 0,\, 0,\, 0].$  Let $x=mn = (4s+3)(4k+2)$ and  that the entries in $A$ cover the
intervals $I_{0 } = [1,6]$ and $I_{8 } = [16sk + 8s + 12k - 29,16sk + 8s + 12k + 6]$.

Now for each $0 \leq r \leq k-2$ let
\[ A_r = \left[\begin{smallmatrix}
4r + 7&-4r - 8&-4r - 9&4r + 10\\
6k - 2r&-10k + 2r + 4&-6k + 2r + 1&10k - 2r - 5\\
-6k - 2r - 1&10k + 2r - 3&6k + 2r + 2&-10k - 2r + 2\\
12k + 16r - 5&-12k - 16r + 4&12k + 16r + 2&-12k - 16r - 1\\
-12k - 16r + 3&12k + 16r - 1&-12k - 16r&12k + 16r - 2\\
12k + 16r + 3&-12k - 16r - 4&-12k - 16r - 5&12k + 16r + 6\\
-12k - 16r - 7&12k + 16r + 8&12k + 16r + 9&-12k - 16r - 10\end{smallmatrix}\right]. \]

Each $A_r$ has all row and column sums equal to zero.  
The entries in $A_r$ cover the
intervals 
$[4r + 7,4r + 10]$, $[6k - 2r - 1,6k - 2r]$, $[6k + 2r + 1,6k + 2r + 2]$, $[10k + 2r - 3,10k + 2r - 2]$, $[10k - 2r - 5,10k - 2r - 4]$ and $[12k + 16r - 5,12k + 16r + 10]$.
Considering $0 \leq r \leq k - 2$, these blocks cover the 
intervals $I_{1 } = [7,4k + 2]$, $I_{2 } = [4k + 3,6n]$, $I_{3 } = [6k + 1,8k - 2]$, $I_{5 } = [10k - 3,12k - 6]$, $I_{4 } = [8k - 1,10k - 4]$ and $I_{6 } = [12k - 5,28k - 22]$.

Additionally, let $E_0$ be a $(4s - 4 \times 4k + 2)$ shiftable Heffter array
and $E = E_0 \pm (28k - 22)$.  Then $sup(E)=I_{7 } = [28k - 21,8{\left(2k + 1\right)} {\left(m - 1\right)} + 28k - 22]$.
Concatenating these intervals, we have covered $[1,16m n + 8m + 12n + 6]=[1,mn]= I_0 I_1 \cdots I_{8}$.

Now, put these all together as follows to construct an array $H$ were
\[H= 
  \begin{array}{|c|c|c|c|} \hline
   A & A_0 &\hspace*{.3in} \cdots\hspace*{.3in} & A_{k-2} \\
   \hline
   \multicolumn{4}{|c|}{E} \\ \hline
  \end{array}
\]
 It is again clear that each row and column sum is 0 modulo $mn +1$ and thus $H$ is an  $H(4s+1,4k+2)$.  This completes the proof.
\end{proof}

%%%%%%%%%%%%%%%%%%%%%%%%%%%%%%%%%%%%%%%%%%%%%%
\section {\label{ells}H($m,n$)  with $m$ and $n$ both odd: the L-constructions}
%%%%%%%%%%%%%%%%%%%%%%%%%%%%%%%%%%%%%%%%%%%%%%

We are now up to the case of $H(m,n)$  with $m$ and $n$ both odd.  As before we still wish to use large shiftable arrays in the construction, however these have both $m$ and $n$ even.  So in order to use these arrays we need to add an odd width border on the top and on the side.  We call this an L-construction. 

The first case is when both $m,n \equiv 1 \pmod 4$, in this case we construct a $H(m,n)$ which will have a border with 9 rows and 9 columns.   More details on how the ingredients were found in each of the next three theorems can again be found in \cite{Tom-thesis}.

\begin{theorem}\label{1x1_mod4.theorem} There exists a Heffter array $H(m,n)$ for all $m,n \equiv 1 \pmod{4}$, $m,n \geq 9$.
\end{theorem}
\begin{proof}
We write $m=4s+1$ and $n=4k+1$ where both $s,k \geq 2$.
Let $x = mn= {\left(4s + 1\right)} {\left(4k + 1\right)}$ and $y = s+k$, 
We begin with the $9 \times 9$ variable  tile which will go in the upper corner. Let

\[ A = \left[\begin{smallmatrix}
x - 17&x - 8&-x + 3&-x + 15&7&8&-9&-10&11\\
x - 7&5&x - 6&-x + 10&x - 1&6y - 5&-10y + 13&-6y + 6&10y - 14\\
4&x - 5&-x&x - 12&-x + 13&-6y + 4&10y - 12&6y - 3&-10y + 11\\
-x + 9&3&-x + 4&x - 14&x - 2&12y - 16&-12y + 15&-12y + 12&12y - 11\\
-x + 11&6&-2&1&x - 16&-12y + 14&12y - 13&12y - 10&-12y + 9\\
12&6y - 7&-6y + 2&12y - 8&-12y + 6&12y + 8&-12y - 15&12y + 3&-12y - 1\\
-13&-10y + 15&10y - 10&-12y + 7&12y - 5&-12y - 10&12y + 14&12y + 13&-12y - 11\\
-14&-6y + 8&6y - 1&-12y + 4&12y - 2&12y + 9&-12y&-12y - 6&12y + 2\\
15&10y - 16&-10y + 9&12y - 3&-12y + 1&-12y - 12&12y + 7&-12y - 5&12y + 4\end{smallmatrix}\right]. \]

Then $A$ has row sums $[0,\, 2x + 1,\, 0,\, 0,\, 0,\, 0,\, 0,\, 0,\, 0]$ and column sums $[0,\, 2x + 1,\, -2x - 1,\, 0,\, 2x+ 1,\, 0,\, 0,\, 0,\, 0].$
 The entries in $A$ cover the
intervals $I_{0 } = [1,15]$, $I_{3 } = [6s + 6k - 8,6s + 6k - 1]$, $I_{6 } = [10s + 10k - 16,10s + 10k - 9]$, $I_{8 } = [12s + 12k - 16,12s + 12k + 15]$ and $I_{11 } = [16s k + 4s + 4k - 16,16s k + 4s + 4k + 1]= [mn-17,mn]$.

For $0 \leq r \leq s + k - 5$, construct a $9 \times 4$ tile \[ B_r = \left[\begin{smallmatrix}
4r + 16&-4r - 17&-4r - 18&4r + 19\\
6s + 6k - 2r - 9&-10s - 10k + 2r + 17&-6s - 6k + 2r + 10&10s + 10k - 2r - 18\\
-6s - 6k - 2r&10s + 10k + 2r - 8&6s + 6k + 2r + 1&-10s - 10k - 2r + 7\\
12s + 12k + 24r + 16&-12s - 12k - 24r - 17&-12s - 12k - 24r - 20&12s + 12k + 24r + 21\\
-12s - 12k - 24r - 18&12s + 12k + 24r + 19&12s + 12k + 24r + 22&-12s - 12k - 24r - 23\\
12s + 12k + 24r + 24&-12s - 12k - 24r - 25&-12s - 12k - 24r - 28&12s + 12k + 24r + 29\\
-12s - 12k - 24r - 26&12s + 12k + 24r + 27&12s + 12k + 24r + 30&-12s - 12k - 24r - 31\\
12s + 12k + 24r + 32&-12s - 12k - 24r - 33&12s + 12k + 24r + 39&-12s - 12k - 24r - 38\\
-12s - 12k - 24r - 35&12s + 12k + 24r + 37&-12s - 12k - 24r - 36&12s + 12k + 24r + 34\end{smallmatrix}\right]. \]

Each $B_r$ has all row and column sums equal to zero.  
 The entries in $B_r$ cover the
intervals $[4r + 16,4r + 19]$, $[6s + 6k + 2r,6s + 6k + 2r + 1]$, $[6s + 6k - 2r - 10,6s + 6k - 2r - 9]$, $[10s + 10k + 2r - 8,10s + 10k + 2r - 7]$, $[10s + 10k - 2r - 18,10s + 10k - 2r - 17]$ and $[12s + 12k + 24r + 16,12s + 12k + 24r + 39]$.
Considering $0 \leq r \leq s + k - 5$, these blocks cover the 
intervals $I_{1 } = [16,4s + 4k - 1]$, $I_{4 } = [6s + 6k,8s + 8k - 9]$, $I_{2 } = [4s + 4k,6s + 6k - 9]$, $I_{7 } = [10s + 10k - 8,12s + 12k - 17]$, $I_{5 } = [8s + 8k - 8,10s + 10k - 17]$ and $I_{9 } = [12s + 12k + 16,36s + 36k - 81]$.

Additionally, let $E_0$ be a $(4s - 8 \times 4k - 8)$ shiftable Heffter integer array
and $E = E_0 \pm (36s + 36k - 81)$, covering the interval $I_{10 } = [36s + 36k - 80,16{\left(k - 2\right)} {\left(s - 2\right)} + 36s + 36k - 81]$.
Concatenating all of these intervals, we have covered $[1,16s k + 4s + 4k + 1] =[1,mn]= I_0 I_1 \cdots I_{11}$. 

In the block construction below it is now clear that each row and column sums to zero modulo $2mn+1$ and hence it is a $(4s + 1 \times 4k + 1)$ Heffter array, i.e. an $H(m,n)$.
\[ 
 \begin{array}{|c|c|}\hline
  A & \begin{array}{c|c|c} B_0 & \cdots & B_{n-3}\end{array} \\
  \hline
  B_{n-2}^T  &  \\
  \hhline{-|~|}
  \vdots&  E\\
  \hhline{-|~|}
  B_{m+n-5}^T &\\ \hline
 \end{array}
\]  \end{proof}

The second case we consider in this section is when $m,n \equiv 3 \pmod{4}$.  In this case we use a border with 7 rows and 7 columns.  

\begin{theorem}\label{3x3_mod4.theorem} There exists a Heffter array $H(m,n)$ for all  $m,n \equiv 3 \pmod{4}$, $m,n \geq 7$.
\end{theorem}
\begin{proof}
  We write $m=4s+3$ and $n=4k+3$ where both $s,k \geq 1$.
Let $x = mn= {\left(4s + 3\right)} {\left(4k + 3\right)}$ and $y = s+k$, 
We begin with the $7 \times 7$  variable tile which will go in the upper corner. Let

\[ A = \left[\begin{smallmatrix}
x - 3&x - 1&5&6&-7&-8&9\\
x&-x + 2&-2&6y + 5&-10y - 5&-6y - 4&10y + 4\\
4&-1&-3&-6y - 6&10y + 6&6y + 7&-10y - 7\\
10&6y + 3&-6y - 8&12y + 14&-12y - 21&12y + 9&-12y - 7\\
-11&-10y - 3&10y + 8&-12y - 16&12y + 20&12y + 19&-12y - 17\\
-12&-6y - 2&6y + 9&12y + 15&-12y - 6&-12y - 12&12y + 8\\
13&10y + 2&-10y - 9&-12y - 18&12y + 13&-12y - 11&12y + 10\end{smallmatrix}\right], \]

Then $A$ has  row sums $[2x + 1,\, 0,\, 0,\, 0,\, 0,\, 0,\, 0]$ and column sums $[2x + 1,\, 0,\, 0,\, 0,\, 0,\, 0,\, 0].$
 The entries in $A$ cover the
intervals $I_{0 } = [1,13]$, $I_{3 } = [6s + 6k + 2,6s + 6k + 9]$, $I_{6 } = [10s + 10k + 2,10s + 10k + 9]$, $I_{8 } = [12s + 12k + 6,12s + 12k + 21]$ and $I_{11 } = [16s k + 12s + 12k + 6,16s k + 12s + 12k + 9]$.

For $0 \leq r \leq s + k - 3$, construct a $7 \times 4$ tile
\[ B_r = \left[\begin{smallmatrix}
4r + 14&-4r - 15&-4r - 16&4r + 17\\
6s + 6k - 2r + 1&-10s - 10k + 2r - 1&-6s - 6k + 2r&10s + 10k - 2r\\
-6s - 6k - 2r - 10&10s + 10k + 2r + 10&6s + 6k + 2r + 11&-10s - 10k - 2r - 11\\
12s + 12k + 16r + 22&-12s - 12k - 16r - 23&12s + 12k + 16r + 29&-12s - 12k - 16r - 28\\
-12s - 12k - 16r - 25&12s + 12k + 16r + 27&-12s - 12k - 16r - 26&12s + 12k + 16r + 24\\
12s + 12k + 16r + 30&-12s - 12k - 16r - 31&-12s - 12k - 16r - 34&12s + 12k + 16r + 35\\
-12s - 12k - 16r - 32&12s + 12k + 16r + 33&12s + 12k + 16r + 36&-12s - 12k - 16r - 37\end{smallmatrix}\right], \]

Each $B_r$ has all row and column sums equal to zero.  
The entries in $B_r$ cover the
intervals $[4r + 14,4r + 17]$, $[6s + 6k - 2r,6s + 6k - 2r + 1]$, $[10s + 10k - 2r,10s + 10k - 2r + 1]$, $[6s + 6k + 2r + 10,6s + 6k + 2r + 11]$, $[10s + 10k + 2r + 10,10s + 10k + 2r + 11]$ and $[12s + 12k + 16r + 22,12s + 12k + 16r + 37]$.
Considering $0 \leq r \leq s + k - 3$, these blocks cover the 
intervals $I_{1 } = [14,4s + 4k + 5]$, $I_{2 } = [4s + 4k + 6,6s + 6k + 1]$, $I_{5 } = [8s + 8k + 6,10s + 10k + 1]$, $I_{4 } = [6s + 6k + 10,8s + 8k + 5]$, $I_{7 } = [10s + 10k + 10,12s + 12k + 5]$ and $I_{9 } = [12s + 12k + 22,28s + 28k - 11]$.

Additionally, let $E_0$ be a $(4s - 4 \times 4k - 4)$ shiftable integer Heffter array
and $E = E_0 \pm (28s + 28k - 11)$, covering the interval $I_{10 } = [28s + 28k - 10,16{\left(k - 1\right)} {\left(s - 1\right)} + 28s + 28k - 11]$.
Concatenating these intervals, we have covered $[1,16s k + 12s + 12k + 9]=[1,mn] = I_0 I_1 \cdots I_{11}$.

In the block construction below it is now clear that each row and column sums to zero modulo $2mn+1$ and hence it is a $(4s + 3 \times 4k + 3)$ Heffter array, i.e. an $H(m,n)$.
\[ 
 \begin{array}{|c|c|}\hline
  A & \begin{array}{c|c|c} B_0 & \cdots & B_{n-2}\end{array} \\
  \hline
  B_{n-1}^T  &  \\
  \hhline{-|~|}
  \vdots&  E\\
  \hhline{-|~|}
  B_{m+n-3}^T &\\ \hline
 \end{array}
\]
\end{proof}

The last case is when $m\equiv 1 \pmod 4$  and  all $n \equiv 3 \pmod{4}$.  Here we will use a border with 9 rows and 7 columns.

\begin{theorem}\label{1x3_mod4.theorem} There exists an integer Heffter array $H(m,n)$ for all  $m\equiv 1 \pmod 4$ with $m \geq 9$ and  all $n \equiv 3 \pmod{4}$ with $n \geq 7$.
\end{theorem}
\begin{proof}
We write $m=4s+1$ and $n=4k+3$ with  $s \geq 2$ and $k \geq 1$.
Let $x = mn= {\left(4s + 1\right)} {\left(4k + 3\right)}$ and $y = s+k$, 
We begin with the $9 \times 7$ variable tile which will go in the upper corner. Let

\[ A = \left[\begin{smallmatrix}
x - 7&3&-x + 4&8&-9&-10&11\\
-x + 3&x - 1&-2&6y + 1&-10y + 3&-6y&10y - 4\\
x - 5&-x&5&-6y - 2&10y - 2&6y + 3&-10y + 1\\
-x + 2&4&x - 6&12y - 4&-12y + 3&-12y&12y + 1\\
7&-6&-1&-12y + 2&12y - 1&12y + 2&-12y - 3\\
12&6y - 1&-6y - 4&12y + 17&12y + 4&-12y - 10&-12y - 18\\
-13&-10y + 5&10y&-12y - 13&12y + 19&-12y - 12&12y + 14\\
-14&-6y + 2&6y + 5&12y + 6&-12y - 8&12y + 16&-12y - 7\\
15&10y - 6&-10y - 1&-12y - 15&-12y - 9&12y + 11&12y + 5\end{smallmatrix}\right] \]

Then $A$  has all row and column sums equal to zero. The entries in $A$ cover the
intervals $I_{0 } = [1,15]$, $I_{5 } = [6s + 6k - 2,6s + 6k + 5]$, $I_{10 } = [10s + 10k - 6,10s + 10k + 1]$, $I_{13 } = [12s + 12k - 4,12s + 12k + 19]$ and $I_{17 } = [16s k + 12s + 4k - 4,16s k + 12s + 4k + 3]$.

For $0 \leq r \leq k - 2$, construct a $9 \times 4$ tile
\[ B_r = \left[\begin{smallmatrix}
4r + 16&-4r - 17&-4r - 18&4r + 19\\
6s + 6k - 2r - 3&-10s - 10k + 2r + 7&-6s - 6k + 2r + 4&10s + 10k - 2r - 8\\
-6s - 6k - 2r - 6&10s + 10k + 2r + 2&6s + 6k + 2r + 7&-10s - 10k - 2r - 3\\
12s + 12k + 24r + 20&-12s - 12k - 24r - 21&-12s - 12k - 24r - 24&12s + 12k + 24r + 25\\
-12s - 12k - 24r - 22&12s + 12k + 24r + 23&12s + 12k + 24r + 26&-12s - 12k - 24r - 27\\
12s + 12k + 24r + 28&-12s - 12k - 24r - 29&-12s - 12k - 24r - 32&12s + 12k + 24r + 33\\
-12s - 12k - 24r - 30&12s + 12k + 24r + 31&12s + 12k + 24r + 34&-12s - 12k - 24r - 35\\
12s + 12k + 24r + 36&-12s - 12k - 24r - 37&12s + 12k + 24r + 43&-12s - 12k - 24r - 42\\
-12s - 12k - 24r - 39&12s + 12k + 24r + 41&-12s - 12k - 24r - 40&12s + 12k + 24r + 38\end{smallmatrix}\right]. \]

Each $B_r$ has all row and column sums equal to zero.  
The entries in $B_r$ cover the
intervals $[4r + 16,4r + 19]$, $[6s + 6k - 2r - 4,6s + 6k - 2r - 3]$, $[6s + 6k + 2r + 6,6s + 6k + 2r + 7]$, $[10s + 10k + 2r + 2,10s + 10k + 2r + 3]$, $[10s + 10k - 2r - 8,10s + 10k - 2r - 7]$ and $[12s + 12k + 24r + 20,12s + 12k + 24r + 43]$.
Considering $0 \leq r \leq k - 2$, these blocks cover the 
intervals $I_{1 } = [16,4k + 11]$, $I_{4 } = [6s + 4k,6s + 6k - 3]$, $I_{6 } = [6s + 6k + 6,6s + 8k + 3]$, $I_{11 } = [10s + 10k + 2,10s + 12k - 1]$, $I_{9 } = [10s + 8k - 4,10s + 10k - 7]$ and $I_{14 } = [12s + 12k + 20,12s + 36k - 5]$.

For $0 \leq r \leq s - 3$, construct a $7 \times 4$ tile
\[ C_r = \left[\begin{smallmatrix}
4k + 4r + 12&-4k - 4r - 13&-4k - 4r - 14&4k + 4r + 15\\
6s + 4k - 2r - 1&-10s - 8k + 2r + 5&-6s - 4k + 2r + 2&10s + 8k - 2r - 6\\
-6s - 8k - 2r - 4&10s + 12k + 2r&6s + 8k + 2r + 5&-10s - 12k - 2r - 1\\
12s + 36k + 16r - 4&-12s - 36k - 16r + 3&12s + 36k + 16r + 3&-12s - 36k - 16r - 2\\
-12s - 36k - 16r + 1&12s + 36k + 16r + 1&-12s - 36k - 16r&12s + 36k + 16r - 2\\
12s + 36k + 16r + 4&-12s - 36k - 16r - 5&-12s - 36k - 16r - 6&12s + 36k + 16r + 7\\
-12s - 36k - 16r - 8&12s + 36k + 16r + 9&12s + 36k + 16r + 10&-12s - 36k - 16r - 11\end{smallmatrix}\right]. \]

Each $C_r$ has all row and column sums equal to zero. The entries in $C_r$ cover the
intervals $[4k + 4r + 12,4k + 4r + 15]$, $[10s + 12k + 2r,10s + 12k + 2r + 1]$, $[6s + 8k + 2r + 4,6s + 8k + 2r + 5]$, $[6s + 4k - 2r - 2,6s + 4k - 2r - 1]$, $[10s + 8k - 2r - 6,10s + 8k - 2r - 5]$ and $[12s + 36k + 16r - 4,12s + 36k + 16r + 11]$.
Considering $0 \leq r \leq s - 3$, these blocks cover the 
intervals $I_{2 } = [4k + 12,4s + 4k + 3]$, $I_{12 } = [10s + 12k,12s + 12k - 5]$, $I_{7 } = [6s + 8k + 4,8s + 8k - 1]$, $I_{3 } = [4s + 4k + 4,6s + 4k - 1]$, $I_{8 } = [8s + 8k,10s + 8k - 5]$ and $I_{15 } = [12s + 36k - 4,28s + 36k - 37]$.

Finally, let  $E_0$ be a $(4s - 8 \times 4k - 4)$ shiftable Heffter array
and $E = E_0 \pm (28s + 36k - 37)$, covering the interval $I_{16 } = [28s + 36k - 36,16{\left(k - 1\right)} {\left(s - 2\right)} + 28s + 36k - 37]$.
Concatenating these intervals, we have covered $[1,16k s + 12k + 4s + 3] =[1,mn]= I_0 I_1 \cdots I_{17}$.

In the block construction below it is now clear that each row and column sums to zero (in $\ZZ$) and hence it is a $(4s + 1 \times 4k + 3)$ integer Heffter array, i.e. an integer $H(m,n)$.
\[ 
 \begin{array}{|c|c|}\hline
  A & \begin{array}{c|c|c} B_0 & \cdots & B_{k-2}\end{array} \\
  \hline
  C_{0}^T  &  \\
  \hhline{-|~|}
  \vdots&  E\\
  \hhline{-|~|}
  C_{s-3}^T &\\ \hline
 \end{array}
\]
\end{proof}

%%%%%%%%%%%%%%%%%%%%%%%%%%%%%%%%%%%%%%%%%%%%%%
\section {Conclusion}\label{Conclusion}
%%%%%%%%%%%%%%%%%%%%%%%%%%%%%%%%%%%%%%%%%%%%%%

In this paper we have constructed  Heffter arrays $H(m,n)$ for all  $m,n\geq 3$. Note that by Theorem
\ref{necessary} an integer $H(m,n)$ does not exist unless $mn \equiv 0,3$ (mod 4) and in all of these cases we did indeed construct an integer Heffter array.   Below is a table showing all of the cases that have been considered in this paper and the theorem covering that case.

\begin{center}
\begin{tabular}{|l|l|} \hline
$m,n$ both even &Theorem \ref{evenbyeventhm} \\ \hline
$m=3$, all $n$& Theorem \ref{3byn_theorem} \\ \hline
$m=5$, all $n$& Theorem \ref{5byn_theorem} \\ \hline
$m$ odd, $n\equiv 0 \pmod 4$& Theorem \ref{oddby4k_theorem} \\ \hline
$m$ odd, $n\equiv 2 \pmod 4$& Theorem \ref{oddby4k+2_theorem} \\ \hline
$m\equiv 1 \pmod 4$, $n \equiv 1 \pmod{4}$ &Theorem \ref{1x1_mod4.theorem} \\ \hline
$m\equiv 3 \pmod 4$, $n \equiv 3 \pmod{4}$ &Theorem \ref{3x3_mod4.theorem} \\ \hline
$m\equiv 1 \pmod 4$, $n \equiv 3 \pmod{4}$ &Theorem \ref{1x3_mod4.theorem} \\ \hline
\end{tabular}
\end{center}

We believe that Heffter arrays are very interesting combinatorial objects. There is also a very useful application of Heffter arrays to biembedding complete graphs and  we think that many of the Heffter arrays constructed in this paper can be used to solve biembedding problems.  A good start towards this is provided in  \cite {DM} where the $H(3,n)$ constructed in Theorem \ref{3byn_theorem}  were reordered and used to prove that for every $v \equiv 1 \pmod6$ there is a biembedding of a Steiner triple system (3-cycle system) and a simple $n$-cycle system both on $v$ points.  We certainly think this can be done for  other $m$ and $n$ based on the Heffter arrays provided in this paper.

%%%%%%%%%%%%%%%%%%%%%%%%%%%%%%%%%%%%%%%%%%%%%%
%%%%%%%%%%%%% ACKNOWLEDGEMENTS %%%%%%%%%%%%%%%%
%%%%%%%%%%%%%%%%%%%%%%%%%%%%%%%%%%%%%%%%%%%%%%

\bigskip \bigskip
\noindent {\bf Addendum:}   Sadly, Dan Archdeacon passed away in February 2015 while this manuscript was in preparation.  We dedicate this paper to his memory.

%%%%%%%%%%%%%%%%%%%%%%%%%%%%%%%%%%%%%%%%%%%%%%
%%%%%%%%%%%%%%% THE REFERENCES %%%%%%%%%%%%%%%%%%
%%%%%%%%%%%%%%%%%%%%%%%%%%%%%%%%%%%%%%%%%%%%%%

%%%%%%%%%%%%%%%%%%%%%%%%%%%%%%%%%%%%%%%%%%%%%%
%%%%%%%%%%%%%%%%%%%%%%%%%%%%%%%%%%%%%%%%%%%%%%
%%%%%%%%%%%%%%%%%%%%%%%%%%%%%%%%%%%%%%%%%%%%%%

\end{document}